\documentclass[11pt]{article}

\usepackage{epsfig}
\usepackage{amssymb, latexsym}
\usepackage{amscd}
\usepackage[all]{xy}
\usepackage{epsf}
\usepackage[mathscr]{eucal}
\usepackage{amsmath,amsfonts,amscd,amssymb,amsthm}
\usepackage{appendix}

\long\def\comment#1\endcomment{}

\theoremstyle{plain}

\newtheorem{theorem}{{\sc Theorem}}[section]
\newtheorem{lemma}[theorem]{\sc Lemma}

\newtheorem{prop}[theorem]{\sc Proposition}
\newtheorem{coroll}[theorem]{\sc Corollary}
\newtheorem{klemma}[theorem]{\sc Key-lemma}

\newcommand{\appsection}[1]{\let\oldthesection\thesection
  \renewcommand{\thesection}{\sc{Appendix \oldthesection}}
  \section{#1}\let\thesection\oldthesection}

\newcommand{\appsubsection}[1]{\let\oldthesubsection\thesubsection
  \renewcommand{\thesubsection}{\sc\oldthesubsection}
  \subsection{#1}\let\thesubsection\oldthesubsection}

\newcommand{\apptheorem}{
\let\oldthetheorem\thetheorem
\renewcommand{\thetheorem}{\sc \oldthetheorem}
\theorem\let\thetheorem\oldthetheorem}

\renewcommand{\thesection}{\sc \arabic {section}}
\renewcommand{\thetheorem}{\thesection.{\sc \arabic {theorem}}}

\theoremstyle{plain}

\newtheorem{defn}[theorem]{\sc Definition}

\theoremstyle{exercise}
\newtheorem{remark}[theorem]{\sc Remark}

\makeatletter \@addtoreset{equation}{section} \makeatother

\def\eqref#1{\thetag{\ref{#1}}}

\let\latexref=\ref
\def\ref#1{{\normalfont{\latexref{#1}}}}

\newcommand{\ldot}{{\:\raisebox{1pt}{\text{\circle*{1.5}}}}}
\newcommand{\udot}{{\:\raisebox{3pt}{\text{\circle*{1.5}}}}}
\def\dlim_#1{{\displaystyle\lim_{#1}}^\hdot}

\newcommand{\gr}{\mathrm{gr}}

\newcommand{\Ext}{\operatorname{Ext}}

\newcommand{\Ker}{\operatorname{{\rm Ker}}}

\newcommand{\Ob}{\mathrm{Ob}}

\newcommand{\Tetra}{\mathscr{T}etra}

\newcommand{\opp}{\mathrm{opp}}

\newcommand{\GS}{\mathrm{GS}}
\newcommand{\Hom}{\mathrm{Hom}}
\newcommand{\Bimod}{\mathscr{B}{imod}}

\newcommand{\RHom}{\mathrm{RHom}}
\newcommand{\Hoch}{\mathrm{Hoch}}

\newcommand{\op}{\mathrm{op}}
\newcommand{\Mon}{{\mathscr{M}}on}
\newcommand{\dg}{\mathrm{dg}}
\newcommand{\Ho}{\mathrm{Ho}}
\newcommand{\Cone}{\mathrm{Cone}}

\newcommand{\id}{\mathrm{id}}

\newcommand{\fint}{{f}}
\newcommand{\cchar}{\mathrm{char}\ }

\newcommand{\SL}{\mathrm{L}}

\newcommand{\Dr}{{\mathscr{D}r}}
\renewcommand{\k}{\Bbbk}
\newcommand{\Alg}{{\mathscr{A}lg}}
\newcommand{\Vect}{{\mathscr{V}}ect^\udot}
\newcommand{\Cat}{{\mathscr{C}at}}

\newcommand{\dprime}{{\prime\prime}}

\title{\sc{Differential graded categories and Deligne conjecture}}

\author{\sc{Boris Shoikhet}}
\date{}

\begin{document}\maketitle

{\footnotesize
\begin{center}{\parbox{4,5in}{{\sc Abstract.}
We prove a version of the Deligne conjecture for $n$-fold monoidal abelian categories $A$ over a field $\k$ of characteristic 0,
assuming some compatibility and non-degeneracy conditions for $A$.
The output of our construction is a weak Leinster $(n,1)$-algebra over $\k$, a relaxed version of the concept of Leinster $n$-algebra in $\Alg(\k)$. The difference between the Leinster original definition and our relaxed one is apparent when $n>1$, for $n=1$ both concepts coincide. 

We believe that there exists a functor from weak Leinster $(n,1)$-algebras over $\k$ to $C_\ldot(E_{n+1},\k)$-algebras, well-defined when $\k=\mathbb{Q}$, and preserving weak equivalences. For the case $n=1$ such a functor is constructed in [Sh4] by elementary simplicial methods, providing (together with this paper) a complete solution for 1-monoidal abelian categories.

Our approach to Deligne conjecture is divided into two parts. The first part, completed in the present paper, provides a construction of a weak Leinster $(n,1)$-algebra over $\k$, out of an $n$-fold monoidal $\k$-linear abelian category (provided the compatibility and non-degeneracy condition are fulfilled). 
The second part (still open for $n>1$) is a passage from weak Leinster $(n,1)$-algebras to $C_\ldot(E_{n+1},\k)$-algebras. 

As an application, we prove in Theorem \ref{3alg} that the Gerstenhaber-Schack complex of a Hopf algebra over a field $\k$ of characteristic 0 admits a structure of a weak Leinster (2,1)-algebra over $\k$ extending the Yoneda structure. It relies on our earlier construction [Sh1] of a 2-fold monoidal structure on the abelian category of tetramodules over a bialgebra.
}}
\end{center}
}

\section{\sc Introduction}
\subsection{\sc Deligne conjecture}
The statement called today ``the classical Deligne conjecture'' was suggested by Pierre Deligne in his 1993 letter to several mathematicians, and currently has several proofs of it, e.g. [MS], [T2], [KS].
It claims the following.
\begin{theorem}\label{thn1}
Let $A$ be an associative algebra (resp., a dg algebra, a dg category) over a field $\k$ of characteristic 0. Then the graded vector space $\RHom^\udot_{\Bimod(A)}(A,A)$ admits a structure of an algebra over the chain operad $C_\ldot(E_2,\k)$ such that the induced action of the homology operad $e_2$ on the Hochschild cohomology $\Ext^\udot_{\Bimod(A)}(A,A)$ is the Gerstenhaber's one [G]. The construction can be performed over $\mathbb{Z}$.
\end{theorem}
Several remarks are in order. The Hochschild cohomology of an associative algebra $A$ is defined intrinsically as $\Ext^\udot_{\Bimod(A)}(A,A)$,
where $\Bimod(A)$ stands for the category of $A$-bimodules. Murray Gerstenhaber found [G] a cup-product $-\cup-$, and a Lie bracket $[-,-]$ of degree -1 (called the Gerstenhaber bracket) on the Hochschild complex of $A$ and proved that the operations $-\wedge-$ and $\{-,-\}$ on the cohomology $H=H(\Hoch^\udot(A))$ they descent to
fulfill the following identities:
\begin{equation}\label{intron1}
\begin{aligned}
\ &\text{(1)   $-\wedge-$ defines an associative commutative structure on $H$},\\
&\text{(2)   $\{-,-\}$ defines a graded Lie algebra structure on $H[1]$},\\
&\text{(3)   $\{a,b\wedge c\}=\{a,b\}\wedge c\pm b\wedge \{a,c\}$  (the Leibniz rule)}
\end{aligned}
\end{equation}
(for any homogeneous $a,b,c\in H$).

Such data is called a {\it Gerstenhaber algebra} over $\k$, or a 2-algebra. The operad of Gerstenhaber algebras is an operad in $\k$-vector spaces, denoted by $e_2$.

In 1976, Fred Cohen [C] proved that the operad $e_2$ is the homology operad of the little discs operad $E_2$, for the case of $\cchar \k=0$:
$e_2=H_\ldot(E_2,k)$.

The situation looked as follows: the cohomology operad of the little discs operad acted on the cohomology of the Hochschild complex .
It motivated Deligne to claim that the chain operad of little discs acts on the Hochschild complex, for any associative algebra $A$.

This claim was highly non-trivial, as the equation (3) of \eqref{intron1} fails on the level of Hochschild cochains:
\begin{equation}\label{intron2}
[\Psi_1,\Psi_2\cup\Psi_3]\ne [\Psi_1,\Psi_2]\cup\Psi_3\pm\Psi_2\cup[\Psi_1,\Psi_3]
\end{equation}
(for homogeneous $\Psi_1,\Psi_2,\Psi_3$).
(Though (2) holds on the Hochschild complex, and (1) holds after the symmetrization).

A proof of Theorem \ref{thn1} was suggested in the Getzler-Jones' 1994 preprint [GJ], but later a mistake in their argument was found.

A new interest to a proof of Deligne conjecture raised up after Tamarkin's 1998 proof [T1] of the Kontsevich formality theorem [Ko], based on operadic methods.
In Tamarkin's proof, the Deligne conjecture plays a central role.
Since that, many new proofs of the Deligne conjecture appeared, see e.g. [MS], [T2], [KS].

Moreover, it was proven that the chain operad of little discs $C_\ldot(E_2,\k)$ is quasi-isomorphic to the operad Koszul resolution $G_\infty$ of the (Koszul) operad $e_2$, and that both operads are quasi-isomorphic to its cohomology (are formal). However, the latter quasi-isomorphisms require transcendental methods. To perform them over $\mathbb{Q}$ one needs to choose a Drinfeld associator over $\mathbb{Q}$.

\subsection{}
In this paper, we prove a generalization of the Deligne conjecture for an $n$-fold monoidal [BFSV] abelian category, which is substantially greater generality (even when $n=1$) than the statement of Theorem \ref{thn1}. The ``output'' in our main Theorem \ref{thn2} is given by some algebraic structure called here {\it a weak Leinster $(n,1)$-algebra}. It is a relaxed version of Leinster monoids introduced in [Le]; more specifically, it is a relaxed version of a Leinster $n$-monoid in the category $\Alg(\k)$ of dg algebras over $\k$. 
We recall the definition of Leinster monoids in Section 2, and give the definition of a weak Leinster $(n,1)$-monoid (whis is seemingly new) in Section \ref{wl}.

Morally, Leinster monoids are closed cousins of weak Segal monoids [Se], for categories enriched over an arbitrary, non necessarily a cartesian-monoidal, symmetric monoidal category. Let $\mathscr{M}$ be a cartesian monoidal category; then Graeme Segal introduced a concept of a {\it weak monoid $M$ in $\mathscr{M}$}. Take the nerve of the monoid $M$, it is a simplicial set $X_\ldot$ with the additional property that the map
$$
X_n\to \underset{n \text{ factors}}{X_1\times X_1\times \dots\times X_1}
$$
is an isomorphism for any $n$, where the map is defined as a successive application of the extreme face maps. The idea was to weaken this property, postulating it to be a ``weak equivalence'', in an appropriate sense.

If we liked to give an analogous definition in the category of $\k$-vector spaces (of complexes of $\k$-vector spaces, of differential graded $\k$-algebras,...) we would immediately see that the above map is ill-defined. If we replaced the cartesian product in the nerve by our product $\otimes$, and set
$$
X_n=\underset{n\text{ factors}}{M\otimes\dots\otimes M }
$$
the corresponding $X_\ldot$ would fail to be a simplicial set (vector space,...). Namely, the extreme face maps are ill-defined.
For two vector spaces $V,W$, there no projections $V\otimes_k W\to V$, $V\otimes_k W\to W$.

The Leinster definition [Le] generalizes the Segal weak monoids for the case when the monoidal category we take the monoids in is not necessarily cartesian-monoidal.

We prove here a version of Deligne conjecture for arbitrary {\it monoidal abelian category}, with weak compatibility of the exact and monoidal structures, see Definition \ref{defn41}. Moreover, we prove it also for an abelian $n$-fold monoidal category, in sense of [BFSV].

The concept of a {\it weak Leinster $(n,1)$-algebra}, which we make use of to provide a non-linear structure on $\RHom^\udot_{\mathscr{A}}(e,e)$ in the statement below, is introduced in Section \ref{wl}.

Our main result is:
\begin{theorem}\label{thn2}
Let $\mathscr{A}$ be an essentially small $\k$-linear abelian $n$-fold monoidal category, where $\cchar \k=0$. Let $e$ be the unit object of $\mathscr{A}$. Assume the weak compatibility of the exact and the monoidal structures, as in Definition \ref{defn41n}. If $n>1$, assume as well that the $n$-fold monoidal structure is non-degenerate, in the sense of Definition \ref{deff}. Then $\RHom^\udot_{\mathscr{A}}(e,e)$ is a weak Leinster $(n,1)$-algebra, whose underlying Leinster 1-algebra product is the Yoneda product in $\RHom^\udot_{\mathscr{A}}(e,e)$.
\end{theorem}
It is proven in Theorem \ref{delignesimple} for the case $n=1$, and in Theorem \ref{delignesimplen} for general $n$.

Note that the non-degeneracy condition of Definition \ref{deff} is apparent only for $n>1$.

The category of $A$-bimodules $\Bimod(A)$ is a $\k$-linear abelian monoidal category, with the monoidal product $M\otimes_A N$, $M,N\in\Bimod(A)$, whose two-sided unit is the tautological $A$-bimodule $A$. Thus the assumptions of Theorem \ref{thn1} is a particular case of those of Theorem \ref{thn2} for $n=1$.

The assumption that $\mathscr{A}$ is essentially small can be appropriately weaken. It can be replaced, for instance, by the assumption that $\mathscr{A}$ is a finitely-presentable Grothendieck category, or that the dg category $\mathscr{A}^\dg$ is generated by a set of perfect objects.
We impose here our assumptions on essential smallness basically to make the presentation more transparent and the main ideas more clear.

\subsection{}
One of the first impacts for the present paper was found in the Kock-To\"{e}n's paper [KT].

The authors prove in loc.cit., in a very conceptual way, a sort of Deligne conjecture for $n$-fold monoidal categories enriched {\it over simplicial sets}.
The authors work with weak Segal monoids, and do not treat the $\k$-linear case. We quote [KT, page 2]:
$${\small
\parbox{5,5in}{
...It is fair to point out that our viewpoint and proof do not seem to work for the original Deligne conjecture, since currently the theory of Segal categories does not work well in linear contexts (like chain complexes), but only in cartesian monoidal contexts. ... However, our original motivation was not to give an additional proof of Deligne's conjecture, but rather to try to understand it from a more conceptual point of view.
}}
$$
This point can be phrased out as follows. Let $M$ be a monoid in $\mathscr{V}ect(\k)$. It is natural to define its ``nerve'' using the tensor product $\otimes_k$ instead of the direct product in the set-enriched case. We define
\begin{equation}
X_n=X\otimes_k X\otimes_k\dots\otimes_k X\ \ \text{($n$ factors)}
\end{equation}
Then $X_\ldot$ is not a simplicial set. The two extreme face maps are ill-defined, as there are no projections $X^{\otimes n}\to X^{\otimes(n-1)}$ along the first (corresp., the last) $n-1$ factors. The origin of the trouble is that the symmetric monoidal category $\mathscr{V}ect(\k)$ is not cartesian-monoidal.

In this paper, we show how to adjust the strategy of [KT] to the $\k$-linear context by making use the Leinster monoids as substitutes of Segal monoids.

The arguments of [KT] rely on the Dwyer-Kan localization, in particular, on a hard result [DK, Corollary 4.7].
Working in the $\k$-linear context, with dg categories over $\k$ instead of simplicial categories, we replace the Dwyer-Kan localization by the Drinfeld construction of dg quotient.
In fact, we need the dg quotient to have some monoidal property, which both Drinfeld's and Keller's constructions of dg quotients fail to have.
As a solution, we construct in Section \ref{section330} a refinement of the Drinfeld dg quotient, having the same homotopy type and improved monoidal properties. Making used the Leinster monoids versus the Segal monoids in [KT], the suitable localization should have a manageable monoidal behavior 
not only on the level of the homotopy category of dg categories, but as well for the dg categories themselves. 
A remarkable feature of the Drinfeld dg quotient, comparably with the earlier Keller's construction of dg quotient, is that it is given by a dg category on the nose, not just an object of the homotopy category of dg categories. Our refinement of the Drinfeld dg quotient, as well as the monoidal property it was designed for, are also defined on the level of dg categories themselves.
We would say that the construction of the refined dg quotient in Section \ref{section330} is technically the main novelty introduced in the paper.

\subsection{\sc Interplay with the J.Lurie approach}
There is another version of higher Deligne conjecture proven in J.Lurie's {\it Higher Algebra} [L1, Ch.5.3], see also [F], [GTZ].

The authors in loc.cit. work with $\infty$-categories, and prove a version of Deligne conjecture for $E_n$-algebras in $\infty$-categories.
To the best of our knowledge, the authors in loc.cit. do not discuss the question of assigning an $E_n$-algebra in $\infty$-categories to a ``classical'' object, namely, to an abelian $n$-fold monoidal category in the sense of [BFSV]. Thus, our main result in Theorem \ref{thn2} seemingly does not follow, at least for $n\ge 2$, from Lurie's approach.

In the same time, we believe that the abelian $n$-fold monoidal categories in the sense of [BFSV] ``appear naturally'' in deformation theory. Our earlier paper [Sh1] shows how to assign a 2-fold monoidal $\k$-linear abelian category with a deformation theory of a bialgebra $B$ over $\k$; however the non-degeneracy (Definition \ref{deff}), which is essential assumption in Theorem \ref{thn2}, does not hold in general. It holds e.g. when $B$ is a Hopf algebra over $\k$, as we show here in Section \ref{stetr}. Any construction of $E_2$-algebras in $\infty$-categories from 2-fold monoidal abelian categories should be sensible to such phenomena.

On the other hand, Theorem \ref{thn2} does not provide a final solution to higher Deligne conjecture; it should be supplemented with a construction of a passage from the weak Leinster $(n,1)$-algebras over $\k$ to $C_\ldot(E_{n+1},\k)$-algebras, which is also a highly non-trivial question. It may use the theory of $\infty$-categories, in particular, on the theory of $\infty$-operads [L1, Ch.2]. The only "elementary" solution we know covers the case $n=1$, see [Sh4]. 

\subsection{\sc Organization of the paper}
In Section 2 we discuss the Leinster monoids in a symmetric monoidal category, following Leinster [Le]. They substitute the Segal weak monoids [Se] for the case when $\mathscr{M}$ is not necessary cartesian-enriched category. In particular, this concept is well-defined in the linear context. Nothing here is new.

In Section \ref{wl} we introduce weak Leinster $(n,1)$-monoids. We arrived to this definition by considering how a based poly-monoidal oplax-functor to $\Cat^\dg(\k)$ is specialized to a functor to $\Alg(\k)$, by taking the $\Hom$s from the based object to itself. This link will become more clear in Section \ref{allthat}, which the reader is advised to read in parallel with Section \ref{wl}.

Section 4 starts with an overview of the Keller and the Drinfeld constructions of dg quotient, as well as of the universal property of a dg quotient. In Sections \ref{section330} and \ref{section33} we introduce a refined ``monoidal'' version of the Drinfeld dg quotient. This construction is the main technical novelty introduced in the paper.

We prove the Deligne conjecture in the form of Theorem \ref{thn2} for $n=1$ in Section 5, and in Section 7 for general $n$. The main results are Theorem \ref{delignesimple} and Theorem \ref{delignesimplen}, for the case $n=1$ and for general $n$, correspondingly.

In Section \ref{allthat} we collect the definitions on monoidal (op)lax-functors, which are necessary for the proof of Deligne conjecture for $n>1$. We also discuss how an $n$-fold monoidal $\k$-linear category gives rise to a strict poly-monoidal (op)lax-functor, see Theorem \ref{monquasi}.

Section 8 contains an application of $n=2$ case of Theorem \ref{thn2} to deformation theory of associative bialgebras. It is a further development of our construction of a 2-fold monoidal structure on the category of tetramodules over a bialgebra $B$, established in [Sh1]. We prove in Theorem \ref{3alg} that, when the bialgebra $B$ is a Hopf algebra, the non-degeneracy condition of Definition \ref{deff} holds for the 2-fold monoidal category of $B$-tetramodules. Then Theorem \ref{thn2} implies that the Gerstenhaber-Schack complex of $B$ (a.k.a. its deformation complex) is a Leinster (2,1)-algebra.

\subsubsection*{\sc Acknowledgements}
The author is grateful to Michael Batanin, Sasha Beilinson, Clemens Berger,
Volodya Hinich, Dima Kaledin, Bernhard Keller, Tom Leinster, Wendy Lowen, Ieke Moerdijk, Stefan Schwede, and to Vadik Vologodsky, for many useful and inspiring discussions. Volodya Hinich found a wrong argument in a proof in Section 6 of the first version of the paper, which is corrected here. 

The current version (2020) fixes an inaccuracy in definition of the category $\mathscr{PC}at^\dg_-(\k)$ in Section 4, which was communicated to the author by Timothy Logvinenko. Here we suggest  an upgrade of the former version of the definition via the Grothendieck construction, which fixes the problem and is conceptually nicer. 
I am thankful to Timothy for his communication and for subsequent correspondence. 

The work was started in 2011-2012, during the author's research stay at the Max-Planck-Institut f\"{u}r Mathematik, Bonn.
I am grateful to the MPIM for the opportunity to stay there, which made possible to interact with many mathematicians visited the MPIM, as well as for the excellent working conditions.

The work was completed in 2015 at the University of Antwerp, where the author's work
was partially supported by the Flemish Science Foundation (FWO) Research grant {\it Krediet aan Navorser} 19/6525.

\vspace{5pt}

{\it Notations:} Throughout the paper, $\k$ denotes a field of characteristic 0. 

\vspace{2pt}

\section{\sc Segal monoids and Leinster monoids}\label{sectionlm}
Here we recall the definitions of a {\it Leinster monoid} and of a {\it Leinster $n$-monoid} in a symmetric monoidal category.

Recall that {\it a Segal monoid} in a symmetric monoidal category $\mathscr{M}$ is a simplicial object $A\colon \Delta^\opp\to\mathscr{M}$ such that the natural maps
$$
\varphi_n\colon A_n\to A_1\times A_1\times\dots\times A_1
$$
defined by the successive application of extreme face maps, are {\it weak equivalences} (in an appropriate sense), where $\times$ denotes the monoidal product in $\mathscr{M}$.

If $M$ is a honest (strict) monoid in $M$, it can be considered as a Segal monoid, with $A_n=M^{\times n}$ (the nerve construction).

This definition makes sense only when $\mathscr{M}$ is a cartesian-monoidal category, that is, when the moinal product is the cartesian product. (For example, the monoidal category $\mathscr{V}ect(\k)$ of vector spaces over a field $\k$, with $\otimes_k$ as the monoidal product, is {\it not} cartesian-monoidal. Indeed, the cartesian product of two vector spaces $V,W$ is their direct sum $V\oplus W$, not the tensor product $V\otimes_k W$).

The matter is that, when $\mathscr{M}$ is not cartesian-monoidal, and $M$ is a honest monoid in $\mathscr{M}$, the nerve $N(M)$ is not, in general, a simplicial set. To make it clear, let us recall the simplicial structure on the nerve of a monoid in the category of sets (or in any other cartesian monoidal category). 

The $n$-simplices of $N(M)$ is the set $A_n=M^{\times n}$. One should define the face maps
$$
F_0,\dots,F_n\colon A_n\to A_{n-1}
$$
and the degeneracy maps 
$$
D_0,\dots, D_n\colon A_n\to A_{n+1}
$$
All but the two extreme face maps are defined as the product of two neighbor factors (the face maps $F_1,\dots,F_{n-1}\colon A_n\to A_{n-1}$ are like that; the face map $F_i$ is defined by the product in $M$ of the $i$-th and $(i+1)$-th factors). 

The {\it two extreme face maps} $F_0,F_n\colon M^{\times n}\to M^{\times (n-1)}$ are defined {\it as the projections} along the rightmost (corresp., the leftmost) factor.

The degeneracy maps are defined as the insertion of the monoidal unit of $M$ to the corresponding place, as a new factor.

Consider as an example the face maps $A_2\to A_1$.
We have $A_1=M$, $A_2=M\times M$. There are three face maps $A_2\to A_1$, corresponding to three possible semi-monotonous maps $[0,1]\to [0,1,2]$ in $\Delta$. The corresponding maps $A_2\to A_1$ are
\begin{equation}
F_0(a\times b)=a,\ \ \ F_1(a\times b)=a*b,\ \ \ F_2(a\times b)=b
\end{equation}
where $*$ is the monoidal product in $M$.

Without the assumption that the monoidal category $\mathscr{M}$ is cartesian-monoidal, only the map $F_1$ among the three maps $F_0,F_1,F_2$ makes sense, as {\it the projections are ill-defined}.

As a conclusion, the nerve $N(M)$ is not a simplicial set, for the case when $\mathscr{M}$ is not cartesian-monoidal. Consequently, the Segal definition of a weak monoid is not an adequate one for this case, as honest monoids fail to be weak ones.

Tom Leinster suggested [Le] the following modification of the Segal definition, which fixes the problem. Let us recall Leinster's definition.

Denote by $\Delta_\fint$ (the category of finite intervals) the subcategory of the simplicial category $\Delta$, having the same objects $[0],[1],[2],\dots$ as $\Delta$, and whose morphisms are the morphisms of $\Delta$ $f\colon [m]\to [n]$ preserving the end-points: $f(0)=0$ and $f(m)=n$. Morally, it is the sub-category in $\Delta$, generated by all degeneracy maps, and by all face maps except the extreme ones.

The category $\Delta_\fint$ is {\it monoidal}(unlike the category $\Delta$ itself). The product is defined on objects as $$[m_1]\otimes [m_2]=[m_1+m_2]$$ where $[m]=\{0<1<\dots<m\}$ is, as usual, a totally ordered set with $m+1$ elements. That is, in the monoidal product we take the quotient-set, by identifying the maximal element of $[m_1]$ with the minimal element of $[m_2]$. Due to the imposed condition on morphisms that they preserve the end-points, the identification extends to gluing the morphisms, making $\Delta_\fint$ a symmetric monoidal category.

\begin{defn}\label{lold}{\rm
\begin{itemize}
\item[(i)]
Let $\mathscr{M}$ be a symmetric monoidal category with a class $T$ of morphisms containing all isomorphisms in $\mathscr{M}$ and closed under the composition. Such a pair $(\mathscr{M}, T)$ is called {\it a monoidal category with weak equivalences} and the morphisms in $T$ are called {\it weak equivalences}.
\item[(ii)]
A {\it Leinster monoid} in a category $\mathscr{M}$ with weak equivalences is a colax-monoidal functor $A\colon\Delta_\fint^\opp\to \mathscr{M}$ whose colax maps $$\beta_{m,n}\colon A_{m+n}\to A_m\otimes A_n$$
and
$$
\alpha\colon A([0])\to \underline{e}
$$
are  weak equivalences
(here $e$ is the unit in $\mathscr{M}$, and $\underline{e}$ is a monoid in $\mathscr{M}$ with single object $e$, and with morphisms $\Hom(e,e)$),
\item[(iii)]
A {\it Leinster pre-monoid} in $\mathscr{M}$ is the same data but dropping the condition that the colax-maps are weak equivalences.
\end{itemize}
}
\end{defn}

\begin{defn}\label{leinster1}{\rm
Let $A$ be a honest monoid in a monoidal category $\mathscr{M}$. The corresponding Leinster monoid $A^\SL\colon\Delta_\fint^\opp\to\mathscr{M}$ is defined as
\begin{equation}\label{slhonest}
A_n^\SL=A\otimes A\otimes\dots \otimes A \text{\ \ \ ($n$ factors) for $n\ge 1$},\ \ A_0=e
\end{equation}
where $e$ is the unit object in $\mathscr{M}$.

The boundary maps $\delta_i\colon A_n^\SL\to A_{n-1}^\SL$
are defined as the product of the $i$-th and $(i+1)$-th factors. The degeneracy maps $\epsilon_i\colon A_n^\SL\to A_{n+1}^\SL$ are defined as the insertion of the unit $e$ to the $i$-th position. (The ill-defined extreme face maps do not belong to $\Delta_\fint$).

This functor has natural colax-monoidal structure, with
$$
\beta_{mn}\colon A_{m+n}^\SL\to A_m^\SL\otimes A_n^\SL
$$
and
$$
\alpha\colon A_0^\SL\to e
$$
to be the identity maps.
}
\end{defn}
One immediately sees
\begin{lemma}\label{n}
Let $\mathscr{M}$ be a monoidal category with weak equivalences, and denote by $\Mon^{\SL}(\mathscr{M})$ the category of Leinster monoids in $\mathscr{M}$. Endow $\Mon^\SL(\mathscr{M})$ with the monoidal structure defined component-wise, by the monoidal structure in $\mathscr{M}$.
Then $\Mon^\SL(\mathscr{M})$ becomes a monoidal category with weak equivalences, whose weak equivalences are the component-wise weak equivalences.
\end{lemma}
\qed

Consider $\mathscr{M}=\mathscr{V}ect^\udot(\k)$, the category of complexes of $\k$-vector spaces, $\k$ is a field. The monoidal product of $V$ and $W$ in $\mathscr{V}ect^\udot(\k)$ is given by their tensor product $V\otimes_\k W$.
\begin{defn}\label{leinster}{\rm
A {\it Leinster 1-algebra} $A^\SL$ over field $\k$ is defined as a Leinster monoid $A^\SL:\Delta_\fint^\opp\to\mathscr{V}ect^\udot(\k)$ in the category $\mathscr{M}=\mathscr{V}ect^\udot(\k)$.
}
\end{defn}

Based on Lemma \ref{n}, we define iteratively the category $\Alg^L_n(\k)$ of Leinster $n$-algebras over $\k$:
\begin{defn}\label{leinstern}{\rm
A Leinster $n$-algebra $A$ is a $\k$-linear Leinster monoid (=Leinster algebra) in the monoidal category $\Alg^L_{n-1}(\k)$ of Leinster $(n-1)$-algebras. It is explicitly given as a functor
$$
(\Delta^\opp_\fint)^{\times n}\to\mathscr{V}ect^\udot(\k),\ \ \ [i_1]\times[i_2]\times\dots\times [i_n]\mapsto A_{i_1i_2\dots i_n}
$$
colax-monoidal by each its argument,
with
$$
A_{11\dots 1}=A
$$
(By a Leinster 0-algebra we understand an element of the category $\Vect(\k)$ itself).
}
\end{defn}

\comment
\subsection{\sc The bar-complex of a Leinster 1-algebra}
Let $F\colon \Delta_f^\opp\to\mathscr{V}ect^\udot(\k)$ be a Leinster 1-algebra. Here we introduce a {\it bar-complex $\mathrm{Bar}^\udot(F)$ of $F$}, which is a genuine dg coalgebra, with some additional property. We do not use this construction in the paper, however, we consider it as a first element of the (future) bridge between Leinter $n$-algebras and homotopy $n$-algebras.
In the case when $X$ is a honest dg algebra over $\k$, and $F([n])=X^{\otimes n}$ is the Segal-Leinster monoid constructed out of $X$, our bar-complex is just the classical (reduced) bar-complex of $X$.

Here is the construction.

We set $Bar^k(F)=F([n])/k=X_n/k$,
where
$k\subset X_n$ is the image of $X_0=k$ obtained by consecutive degeneracy maps. In the case when $X_n=X^{\otimes n}$ for a dg algebra $X$, this subspace $k\subset X^{\otimes n}$ is the 1-dimensional subspace generated by $1\otimes\dots\otimes 1$ ($n$ factors).

We use the notation
\begin{equation}
X_n^+=X_n/k
\end{equation}

The differential is
\begin{equation}
d_n\colon X_n^+\to X_{n-1}^+\ \ d_n=\sum_{i=1}^{n-1}\delta_i
\end{equation}
where $\delta_i$ is the $i$-th face map. Here these maps are $\delta_1,\dots,\delta_{n-1}$, they belong to $\Delta_f$, but the two extreme ones, $\delta_0$ and $\delta_n$, do not. Our observation is that in the classical bar-complex of an associative algebra, the differential is written using only the face maps from
$\Delta_f^\opp$.

At the next step, we define a coassociative coproduct on $\mathrm{Bar}^\udot(F)$.
It is given by $\omega\colon \mathrm{Bar}^\udot(F)\to\mathrm{Bar}^\udot(F)^{\otimes 2}$, where
$\omega|_{X_n}$ is the obtained from the map
\begin{equation}\label{nonreduced}
\omega_n^\prime\colon X_n\xrightarrow{\oplus\beta_{ij}}\bigoplus_{i+j=n}X_i\otimes X_j
\end{equation}
given by the direct sum of the colax maps $\beta_{ij}$.
The coassociativity of $\omega$ follows from the colax-equations for $\beta_{ij}$s, and the compatibility of $\omega$ with the differential follows from the fact the $\beta_{ij}$ are colax maps for the {\it functor} $F$, and, therefore, are compatible with all faces and degeneracies maps.
We need the compatibility with the face maps for the compatibility with the differential, and we need the compatibility with the degeneracy maps to prove that $\omega_n^\prime$ descents from \eqref{nonreduced} to
\begin{equation}\label{reduced}
\omega_n\colon X_n^+\xrightarrow{\oplus\beta_{ij}}\bigoplus_{i+j=n}X_i^+\otimes X_j^+
\end{equation}

So far, we did not use the property that in a Leinster 1-algebra the colax-map $\beta_{ij}$ are quasi-equivalences of complexes. This property is used to establish Proposition \ref{privet} below, distinguishing  the dg coalgebras $\mathrm{Bar}^\udot(F)$ for Leinster monoids $F$, among all dg coalgebras. We need some preparation to formulate it.

Recall the Quillen adjunction
\begin{equation}\label{adjunction}
L\colon\ \mathscr{C}oalg^\udot_+(\k)\rightleftarrows \mathscr{A}lg_+^\udot(\k)\ \colon R
\end{equation}

Here $\mathscr{A}lg^\udot(\k)$ is the closed model category of $\mathbb{Z}$-graded dg algebras over $\k$, and $\mathscr{C}oalg^\udot_+(\k)$ is the closed model category of non-counital pro-nilpotent dg coalgebras over $\k$. A coalgebra $C$ is called {\it pro-nilpotent} if the ascending sequence of its subspaces
$$
C_n=\Ker\{\Delta^n \colon C\to C^{\otimes (n+1)}\}
$$
forms an exhaustive ascending filtration of $C$. The corresponding closed model structures in such generality were constructed by Hinich [Hi1,2].

In adjunction \eqref{adjunction}, the functor $L$ applied to a coalgebra $C$ is the bar-complex of $C$, and the functor $R$ applied to an algebra $A$
is the {\it reduced} bar-complex of $A$.

\begin{prop}\label{privet}
For any Leinster monoid $F$, the bar-complex $\mathrm{Bar}^\udot(F)$ is a dg coalgebra, enjoying the following additional property: the adjunction map
$\mathrm{Bar}^\udot(F)\to R\circ L (\mathrm{Bar}^\udot(F))$ is quasi-isomorphism of dg coalgebras.
\end{prop}

\endcomment

\section{\sc Weak Leinster $(n,1)$-monoids}\label{wl}
Here we introduce some relaxed version of the special Leinster $n+1$-monoids wich are Leinster $n$-monoids in strict 1-monoids. To simplify the notations, we give the definition only for the case of monoids in $\Vect(\k)$.

For $n=1$, the concept of a weak Leinster $(1,1)$-monoid agrees with the concept of usual 1-Leinster monoids in $\Alg(\k)$, see Definition \ref{lold}. However, it is essentially more relaxed for $n>1$.
We had arrived to it in our study of the Deligne conjecture for $n$-fold monoidal abelian categories, for $n>1$.
The reader will see that the way we make the Definition \ref{leinstern} relaxed comes from the categorical concept of a strict poly-monoidal oplax-functor, see Definitions \ref{quasi}, \ref{quasibis}, \ref{bifmon} below.

Note that in all our examples the elements $g(\underline{m})$ (see (4)) are {\it equal} to the monoid units.

\begin{defn}\label{leinsterrelaxed}
{\rm
A {\it weak Leinster $(n,1)$-monoid} in $\Vect(\k)$ is data which assigns:
\begin{itemize}
\item[(1)] to each object $[m_1]\times\dots \times [m_n]$ of the category $(\Delta_f^\opp)^{\times n}$ a strict algebra $X_{m_1,\dots,m_n}\in\Vect(\k)$ (that is, $X_{m_1,\dots,m_n}$ is a dg algebra),
\item[(2)] to each morphism $\alpha\colon [m_1]\times\dots \times [m_n]\to [\ell_1]\times\dots\times [\ell_n]$ a strict morphism of monoids (that is, a map of dg algebras)
$$
\alpha_*\colon X_{m_1,\dots,m_n}\to X_{\ell_1,\dots,\ell_n}
$$
\item[(3)] for any two composable morphisms $\alpha\colon [m_1]\times\dots \times [m_n]\to [\ell_1]\times\dots\times [\ell_n]$
and $\beta\colon [\ell_1]\times\dots\times [\ell_n]\to [p_1]\times\dots\times [p_n]$ an element 
$$
g(\alpha,\beta)\in X_{p_1,\dots,p_n}
$$
{\it It is a cycle of degree 0 in the dg algebra $X_{p_1,\dots,p_n}$, which descents to the unit of the cohomology algebra $H^\udot(X_{p_1,\dots,p_n})$},
\item[(4)] to each object $\underline{m}=[m_1]\times\dots\times [m_n]\in (\Delta_f^\opp)^{\times n}$, an element
$$
g(\underline{m})\in X_{m_1,\dots,m_n}
$$
{\it It is a cycle of degree 0 in the dg algebra $X_{m_1,\dots,m_n}$, which descents to the unit of the cohomology algebra $H^\udot(X_{m_1,\dots,m_n})$}, 
\item[(5)] maps of dg algebras which are quasi-equivalences in $\Vect(\k)$
$$
\theta_{m_1,\dots,(m_s,m_s^\prime), m_{s+1},\dots,m_n}\colon X_{m_1,\dots,m_s+m_s^\prime,\dots,m_n}\to
X_{m_1,\dots,m_s,\dots,m_n}\otimes X_{m_1,\dots,m_s^\prime,\dots,m_n}
$$
defined for all $m_1,\dots,m_s,m_s^\prime,\dots,m_n$, 
\end{itemize}
subject to the following conditions:
\begin{itemize}
\item[(i)] for any two composable $\alpha,\beta$ in $(\Delta_f^\op)^{\times n}$ and any $x\in X_{m_1,\dots,m_n}$, one has
\begin{equation}
(\beta_*\alpha_*(x))\circ g(\alpha,\beta)=g(\alpha,\beta)\circ((\beta\alpha)_*(x))
\end{equation}
where $\circ$ is the product in the monoid $X_{p_1,\dots,p_n}$,
\item[(ii)] $g(\alpha,\beta)$ is equal to the unit of $X_{p_1,\dots,p_n}$ if both $\alpha$ and $\beta$ act on the same factor $\Delta_f^\opp$ in $(\Delta_f^\opp)^{\times n}$, keeping the remaining $n-1$ factors fixed,
\item[(iii)] for any three composable arrows $\alpha\colon [m_1]\times\dots\times [m_n]\to [\ell_1]\times \dots\times[\ell_n]$,
$\beta\colon [\ell_1]\times\dots\times [\ell_n]\to [p_1]\times\dots\times[p_n]$, $\gamma\colon 
[p_1]\times\dots\times[p_n]\to [q_1]\times\dots\times [q_n]$, and for any $x\in X_{m_1,\dots,m_n}$, one has
the following equality in $X_{q_1,\dots,q_n}$:
\begin{equation}
g(\beta,\gamma)\circ g(\alpha,\gamma\beta)=
\gamma_*(g(\alpha,\beta))\circ g(\beta\alpha,\gamma)
\end{equation}
where $\circ$ is the product in $X_{q_1,\dots,q_n}$,
\item[(iv)] for any $\alpha\colon \underline{m}:=[m_1]\times\dots \times [m_n]\to [\ell_1]\times\dots\times [\ell_n]:=\underline{\ell}$ one has:
\begin{equation}
g(\underline{\ell})\circ g(1,\alpha)=1_{X_{\ell_1,\dots,\ell_n}},\ \  \alpha_*(g(\underline{m}))\circ g(\alpha,1)=1_{X_{\ell_1,\dots,\ell_n}}
\end{equation}
\item[(v)] the map $\theta_{m_1,\dots,m_{i-1},(?,?),m_{i+1},\dots,m_n}$ is colax-monoidal 
for any $i$ and for any fixed \\ $m_1,\dots,m_{i-1},m_{i+1},\dots,m_n$, 
\item[(vi)]
for any morphisms in $\Delta_f^\opp$ $\alpha_j\colon [m_j]\to [\ell_j]$ for $j\ne i$, and $\alpha_i^\prime\colon [m_i^\prime]\to [\ell_i^\prime],\alpha_i^\dprime \colon [m_i^\dprime]\to[\ell_i^\dprime]$ in $\Delta_f^\opp$, the diagram below commutes
\begin{equation}
\xymatrix{
X_{m_1,\dots,m_i^\prime+m_i^\dprime,\dots,m_n}\ar[rr]^{{\theta\hspace{25pt}}}\ar[d]_{(f^\otimes)_*}&&X_{m_1,\dots,m_i,\dots,m_n}\otimes X_{m_1,\dots,m_i^\dprime,\dots,m_n}\ar[d]^{(f^\prime)_*\otimes (f^\dprime)_*}\\
X_{\ell_1,\dots,\ell_i^\prime+\ell_i^\dprime,\dots,\ell_n}\ar[rr]^{{\theta\hspace{25pt}}}&&X_{\ell_1,\dots,\ell_i^\prime,\dots,\ell_n}\otimes X_{\ell_1,\dots,\ell_i^\dprime,\dots,\ell_n}
}
\end{equation}
where $f^{\otimes},f^\prime,f^\dprime$ are the morphisms in $(\Delta_f^\opp)^{\times n}$, given by formulas
$$
f^{\otimes}=(\alpha_1,\dots,\alpha_i^\prime\otimes\alpha_i^\dprime,\dots,\alpha_n),\ \ f^\prime=
(\alpha_1,\dots,\alpha_i^\prime,\dots,\alpha_n),\ \ f^\dprime=(\alpha_1,\dots,\alpha_i^\dprime,\dots,\alpha_n)
$$
\item[(vii)] $X_{m_1,\dots,m_n}=\k$, if $m_i=0$ for some $i$,
\item[(viii)] for any $1\le i\le n$, and for any $m_1,\dots,m_n$, the composition 
\begin{equation}
X_{m_1,\dots,m_n}\xrightarrow{\theta}X_{m_1,\dots,m_i,\dots,m_n}\otimes X_{m_1,\dots,\underset{i\text{-th
}}{0},\dots,m_n}\xrightarrow{\mathrm{(vii)}}X_{m_1,\dots,m_n}
\end{equation}
is the identity map.

\end{itemize}

}
\end{defn}
\begin{remark}{\rm
One can easily see from (vi), as a very particular case, that the diagram below commutes:
\begin{equation}
\xymatrix{
X_{m_1,\dots,m_i,\dots m_n}\otimes X_{m_1,\dots,m_i^\prime,\dots,m_n}\ar[d]_{\id_*\otimes (\alpha_i^\prime)_*}&X_{m_1,\dots,m_i+m_i^\prime,\dots,m_n}\ar[l]\ar[dd]_{(\alpha_i\otimes\alpha_i^\prime)_*}\ar[r]&X_{m_1,\dots,m_i,\dots m_n}\otimes X_{m_1,\dots,m_i^\prime,\dots,m_n}
\ar[d]^{(\alpha_i)_*\otimes\id_*}\\
X_{m_1,\dots,m_i,\dots m_n}\otimes X_{m_1,\dots,\ell_i^\prime,\dots,m_n}\ar[d]_{(\alpha_i)_*\otimes\id_*}&&X_{m_1,\dots,\ell_i,\dots m_n}\otimes X_{m_1,\dots,m_i^\prime,\dots,m_n}\ar[d]^{\id_*\otimes(\alpha_i^\prime)_*}\\
X_{m_1,\dots,\ell_i,\dots m_n}\otimes X_{m_1,\dots,\ell_i^\prime,\dots,m_n}&X_{m_1,\dots,\ell_i+\ell_i^\prime,\dots,m_n}\ar[r]\ar[l]&X_{m_1,\dots,\ell_i,\dots m_n}\otimes X_{m_1,\dots,\ell_i^\prime,\dots,m_n}
}
\end{equation}
where the horizontal arrows are the maps $\theta$'s introduced in (5).
}
\end{remark}

\section{\sc The Keller's and Drinfeld's constructions of dg quotient}
We refer the reader to [K2] for the definition and basic facts on differential graded (dg) categories.
\subsection{\sc The dg quotient of dg categories}
The dg quotient $\mathscr{C}/\mathscr{C}_0$ of a dg category $\mathscr{C}$ by an {\it essentially small} full dg subcategory $\mathscr{C}_0$ was firstly introduced by Bernhard Keller in [K1, K3]. It is a dg category, defined uniquely up to a quasi-equivalence. When $\mathscr{C}$ is a pre-triangulated dg category, and $\mathscr{C}_0$ its full essentially small pre-triangulated dg subcategory, it has the following property.

\begin{prop}[B.Keller, {[K3, Section 4]}]\label{keller}
Let $\mathscr{C}$ be a pre-triangulated dg category, $\mathscr{C}_0$ its essentially small full pre-triangulated dg sub-category.
Then is a pre-triangulated dg category $\mathscr{C}/\mathscr{C}_0$ whose triangulated category
\begin{equation}
H^0(\mathscr{C}/\mathscr{C}_0)\simeq H^0(\mathscr{C})/H^0(\mathscr{C}_0)
\end{equation}
where the quotient in the right-hand side is the Verdier quotient of triangulated categories.
\end{prop}
The dg quotient is a functor from the category of pairs of dg categories $(\mathscr{C},\mathscr{C}_0)$ with $\mathscr{C}_0$ essentially small full subcategory, to the homotopy category of dg categories, which can be characterized by a universal property (see below).
In the case when $\mathscr{C}$ is pre-triangulated, the dg quotient $\mathscr{C}/\mathscr{C}_0$ has the same image in the homotopy category of dg categories as the To\"{e}n dg localization [To1, Section 8.2] $\mathscr{C}[S^{-1}]$ where $S$ is the set of closed degree 0 morphisms $s$ in $\mathscr{C}$ such that $\mathrm{Cone}(s)\in\mathscr{C}_0$.

\smallskip

V.Drinfeld [Dr2] provided another construction of the dg quotient $\mathscr{C}/\mathscr{C}_0$ (which we recall in Section \ref{drq}).
It is beneficial by being defined as a honest dg category, not just as an object of the homotopy category of dg categories.
It has the same objects as $\mathscr{C}$, and its morphisms are obtained as some free envelope of the morphisms in $\mathscr{C}$ with newly added morphisms of degree -1, that kill up to homotopy the objects in $\mathscr{C}_0$.

In this Section, we provide a refinement of the Drinfeld construction of dg quotient, which has the same homotopy type, but enjoys a more manageable behavior with respect to the tensor product, see Proposition \ref{gendr} below. 

Drinfeld formulated in [Dr2] a universal property, which characterizes a dg quotient uniquely, up to an isomorphism, as an objects of the homotopy category of dg categories $\Ho\mathscr{C}at^\dg(\k)$.
Tabuada [Tab] re-considered the question on the universal property of Drinfeld's dg quotient and proved a refined version of it; the result below is due to Tabuada [Tab, Theorem 4.0.1].
Denote by $[\mathscr{X}]$ the object of the homotopy category (the localization by quasi-equivalences) $\Ho\Cat^\dg(\k)$, corresponded to a dg category $\mathscr{X}$.
The universal property of dg quotient reads:
\begin{theorem}[Drinfeld, Tabuada]\label{druniv}
Let $\mathscr{C}\supset\mathscr{C}_0$ dg categories, with $\mathscr{C}_0$ essentially small.
The morphisms $F\colon [\mathscr{C}]\to[\mathscr{D}]$ in $\Ho\mathscr{C}at^\dg(\k)$ such that the corresponding functor $H^0F\colon H^0[\mathscr{C}]\to H^0[\mathscr{D}]$ of homotopy categories of dg categories sends the image of $H^0[\mathscr{C}_0]$ in $H^0[\mathscr{C}]$ to 0, are in 1-to-1 correspondence with the morphisms  $\overline{F}\colon [\mathscr{C}/\mathscr{C}_0]\to[\mathscr{D}]$ in $\Ho\mathscr{C}at^\dg(\k)$. {\rm (}To say that $H^0(F)$ maps $H^0(\mathscr{C}_0)$ to 0 means, by definition, that  $H^0F(\id_X)$ for any $X$ in the image of $H^0[\mathscr{C}_0]$ in $H^0[\mathscr{C}]$, is zero morphism in $H^0[\mathscr{D}]${\rm )}.
\end{theorem}
Drinfeld [Dr2] and Tabuada proved [Tab] proved that both constructions of dg quotient, the one of Keller and the one of himself, fulfil this universal property. Therefore, they define isomorphic objects of the homotopy category $\Ho\mathscr{C}at^\dg(\k)$. We make an essential use of this result; it plays in our paper the role similar to the Dwyer-Kan result [DK, Corollary 4.7] in the Kock-To\"{e}n non-linear Deligne conjecture [KT].

\subsection{\sc Drinfeld dg quotient}\label{drq}
Let $\mathscr{C}$ be a dg category over a field $\k$, and let $\mathscr{C}_0$ be its essentially small full dg subcategory. Drinfeld defines [Dr] the dg quotient
$\mathscr{C}/\mathscr{C}_0$ as follows.

The category $\mathscr{C}/\mathscr{C}_0$ has the same objects as $\mathscr{C}$, and the category $\mathscr{C}$ is embedded into $\mathscr{C}/\mathscr{C}_0$ as a dg category. Choose an object $X$ in $\mathscr{C}_0$ for any class of isomorphism of objects in $\mathscr{C}_0$, these objects $\{X\}$ define a sub-category $\overline{\mathscr{C}_0}\subset\mathscr{C}_0$ . The assumption that $\mathscr{C}_0$ is essentially small guarantees that $\overline{\mathscr{C}}_0$ is small. For any object $X$ in $\overline{\mathscr{C}_0}$, a new morphism $\varepsilon_X$ in $\Hom(X,X)$ of degree -1, with $d(\varepsilon_X)=\id_X$, is added, without any relations. By definition, $\mathscr{C}/\mathscr{C}_0$ is the category with the objects $\{\Ob\mathscr{C}\}$, and whose morphisms are obtained as the free algebraic envelope of $\{\varepsilon_X\}_{X\in\mathscr{C}_0}$ with the morphisms of $\mathscr{C}$.

More precisely, for any $X,Y\in \Ob\mathscr{C}$, the underlying graded $\k$-vector space of morphisms is $$\Hom_{\mathscr{C}/\mathscr{C}_0}(X,Y)=\oplus\Hom_{\mathscr{C}/\mathscr{C}_0}^{(n)}(X,Y)$$
where
\begin{equation}\label{drinmor}
\Hom_{\mathscr{C}/\mathscr{C}_0}^{(n)}(X,Y)=\oplus_{Y_0,\dots,Y_{n-1}\in\mathscr{C}_0}\Hom_{\mathscr{C}}(X,Y_0)\otimes k[1]\otimes
\Hom_{\mathscr{C}}(Y_0,Y_1)\otimes k[1]\otimes \dots\otimes k[1]\otimes \Hom_{\mathscr{C}}(Y_{n-1},Y)
\end{equation}
where the $i$-th factor $k[1]$ is spanned by $\varepsilon_{Y_{i+1}}$. Here in \eqref{drinmor} some of the objects $Y_i$s may coincide.

The differential maps $\Hom_{\mathscr{C}/\mathscr{C}_0}^{(n)}$ to $\Hom_{\mathscr{C}/\mathscr{C}_0}^{(n-1)}$, and the category $\mathscr{C}/\mathscr{C}_0$ is endowed with an ascending filtration.

\comment
\begin{remark}{\rm
One easily sees that the category $\mathscr{C}/\mathscr{C}_0$ is quasi-equivalent to the category $(\mathscr{C}/\mathscr{C}_0)^\prime$ which is the free product of the category $\mathscr{C}$ with the morphisms $\{\varepsilon_X,\ X\in\mathscr{C}_0,\ d\varepsilon_X=\id_X\}$. (The difference between the two ones is that the {\it iterated} compositions like $\varepsilon_X^{\circ m}$ do not vanish in $(\mathscr{C}/\mathscr{C}_0)^\prime$). The quasi-equivalence holds as the complex
\begin{equation}
\dots k[1]^{\otimes n}\to k[1]^{\otimes (n-1)}\to\dots\to k[1]^{\otimes 2}\to k[1]\to k\to 0
\end{equation}
where the differentials are defined from $d(\varepsilon)=1$ (where $\varepsilon=1[1]$) by the odd Leibniz rule, is quasi-isomorphic to the complex
\begin{equation}
0\to k[1]\to k\to 0
\end{equation}
which is in fact used in the Drinfeld's construction. That is, one can think on the Drinfeld's category $\mathscr{C}/\mathscr{C}_0$ as on the simplified for computations quasi-equivalent free construction $(\mathscr{C}/\mathscr{C}_0)^\prime$.
}
\end{remark}
\endcomment

The dg category $\mathscr{C}/\mathscr{C}_0$ does not depend, up to a quasi-equivalence, on the choice of small dg subcategory $\overline{\mathscr{C}_0}$. Moreover, different choices of $\overline{\mathscr{C}_0}$ result in {\it equivalent} (not just quasi-equivalent) dg categories.

It implies that we have a {\it functor}
\begin{equation}
\mathscr{P}_1\mathscr{C}at^\dg\to\mathscr{C}at^\dg
\end{equation}
from the category $\mathscr{P}_1\mathscr{C}at^\dg$ of pairs $(\mathscr{C},\mathscr{C}_0)$ with $\mathscr{C}_0$ essentially small, to the category $\mathscr{C}at^\dg$ (not just to the homotopy category $\Ho\mathscr{C}at^\dg$).

\subsection{\sc Drinfeld's dg quotient: a refinement}\label{section330}
Let $\mathscr{C}$ be a dg category, and let $\mathscr{C}_1,\mathscr{C}_2,\dots,\mathscr{C}_k$ be its full essentially small dg subcategories.
Here we construct a dg category
$$
\mathscr{C}/(\mathscr{C}_1,\dots,\mathscr{C}_k)
$$
called {\it the refined Drinfeld dg quotient}.

First of all, we replace the essentially small full categories $\mathscr{C}_i$ by small full categories $\overline{\mathscr{C}}_i$, taking an object from each isomorphism class of objects in $\mathscr{C}_i$. We take care that the chosen objects agree for all intersections $\mathscr{C}_i\cap\mathscr{C}_j$, and thus the categories
 \begin{equation}\label{choicesxxx}
\overline{\mathscr{C}}_{i_1}\cap\dots\cap\overline{\mathscr{C}}_{i_\ell}\sim \mathscr{C}_{i_1}\cap\dots\cap\mathscr{C}_{i_\ell}
\end{equation}
are equivalent.
It is clear that it is always possible to achieve.

For any $X=X_i\in \overline{\mathscr{C}}_i$, $i=1,\dots,k$, we add formally an element $\varepsilon^i_X$ which is a morphism from $X$ to $X$ of degree -1, with the differential $d\varepsilon^i_X=\id_X$.

For any $X=X_{ij}\in\overline{\mathscr{C}}_i\cap\overline{\mathscr{C}}_j$, $i<j$, we introduce formally a morphism $\varepsilon^{ij}_X$ from $X$ to itself of degree -2, with $d\varepsilon^{ij}_X=\varepsilon^i_X-\varepsilon^j_X$.

For any $X=X^{ijk}\in\overline{\mathscr{C}}_i\cap\overline{\mathscr{C}}_j\cap\overline{\mathscr{C}}_k$, $i<j<k$, we introduce formally a morphism $\varepsilon_X^{ijk}$ from $X$ to itself of degree -3, with $d\varepsilon^{ijk}_X=\varepsilon^{ij}_X-\varepsilon^{ik}_X+\varepsilon^{jk}_X$, and so on.

The sign rule is as in the \v{C}ech cohomology theory, which implies $d^2=0$.

Now the dg category $\mathscr{C}/(\mathscr{C}_1,\dots,\mathscr{C}_k)$ has the same objects as the dg category $\mathscr{C}$, and the morphisms are freely generated by the morphisms in $\mathscr{C}$ with the given composition among them, and by the newly added morphisms $\varepsilon_X^{i_1i_2\dots i_s}$ of degree $-s$, with the differentials of $\varepsilon_X^{i_1,\dots,i_s}$ defined as above and extended to the whole morphisms by the Leibniz rule. 

Denote by $\mathscr{C}_\Sigma$ the full dg subcategory of $\mathscr{C}$ having the objects
$$
\Ob\mathscr{C}_\Sigma=\bigcup_{i=1}^k\Ob\overline{\mathscr{C}}_i
$$

Consider the Drinfeld dg quotient $\mathscr{C}/\mathscr{C}_\Sigma$.
There is a natural dg functor
\begin{equation}\label{psi}
\Psi\colon\mathscr{C}/(\mathscr{C}_1,\dots,\mathscr{C}_k)\to\mathscr{C}/\mathscr{C}_\Sigma
\end{equation}
sending all $\varepsilon_X^i$ to $\varepsilon_X$, for $X\in\Ob\mathscr{C}_\Sigma$, and sending all $\varepsilon_X^{i_1\dots i_s}$ to 0, for $s>1$.

\begin{lemma}\label{angle}
In the above notations, the map $\Psi$ is a quasi-equivalence of dg categories.
\end{lemma}
\begin{proof}
The assertion easily follows from the two following elementary remarks.

Let $V$ be a vector space of dimension $n$, with fixed basis $\{e_1,\dots,e_n\}$. Consider the complex $\Lambda^\udot_d(V)$, having the component $\Lambda^\ell V$ in degree $-\ell$, with the differential
of degree $+1$, defined as
\begin{equation}\label{lambda}
d(e_{i_1}\wedge\dots\wedge e_{i_\ell})=\sum_{s=1}^\ell(-1)^{s-1}e_{i_1}\wedge\dots\wedge\hat{e_{i_s}}\wedge\dots\wedge e_{i_\ell}
\end{equation}
Then the complex $\Lambda^\udot_d(V)$ is acyclic in all degrees.

The latter claim is clear: there is the homotopy operator $h\colon\Lambda^\ell V\to \Lambda^{\ell+1}V$,
$h(\omega)=(e_1+e_2+\dots+e_n)\wedge\omega$, with $[d,h](\omega)=\pm \omega$.

Consider the complex $\Lambda^\udot_d(U_1)$, for $U_1$ the 1-dimensional vector space over $\k$, with a choused basis vector $e$. The natural map
$\Psi_V\colon\Lambda^\udot_d(V)\to\Lambda^\udot_d(U_1)$ which maps $\Lambda^\ell V$ to 0 for $\ell>1$, and with maps each basis vector $e_i$ in $V$ to the basis vector $e$ in $U_1$. It is a quasi-isomorphism of (acyclic) complexes.
\end{proof}

\subsection{\sc The monoidal property of the generalized Drinfeld dg quotient}\label{section33}
Define the category $\mathscr{C}at^\dg_-(\k)$ as the category of small $\k$-linear dg categories $\mathscr{C}$, such that for any $X,X^\prime\in\Ob\mathscr{C}$, the graded components $\Hom^i(X,X^\prime)=0$ for $i>N(X,X^\prime)$, for some integral number $N(X,X^\prime)$ depending on $X,X^\prime$.

The category $\mathscr{C}at^\dg_-(\k)$ is a monoidal category. The monoidal product $\mathscr{C}\otimes\mathscr{D}$ of two categories $\mathscr{C},\mathscr{D}\in\mathscr{C}at^\dg_-(\k)$ has objects $\Ob\mathscr{C}\times\Ob\mathscr{D}$,
and
\begin{equation}
\Hom_{\mathscr{C}\times\mathscr{D}}(X\times Y,X^\prime\times Y^\prime)=\Hom_{\mathscr{C}}(X,X^\prime)\otimes_k\Hom_{\mathscr{D}}(Y,Y^\prime)
\end{equation}
The bounding condition on the morphism complexes is imposed to make the tensor product in the right-hand side well-behaved.

\smallskip

Introduce the category $\mathscr{PC}at^\dg_-(\k)$, as follows.

We firstly define categories $\mathscr{P}_n\mathscr{C}at^\dg_{-}(\k)$, $n\ge 0$.
An object of $\mathscr{P}_n{C}at_-^\dg(\k)$ is a pair $(\mathscr{C};\  \mathscr{C}_1,\dots,\mathscr{C}_n)$,
where $\mathscr{C}\in\mathscr{C}at^\dg_-(\k)$, and the second argument is an ordered $n$-tuple of its {\it small} full dg subcategories $\mathscr{C}_1,\dots,\mathscr{C}_n$.

A morphism $f\colon (\mathscr{C};\ \mathscr{C}_1,\dots,\mathscr{C}_n)\to
(\mathscr{D};\ \mathscr{D}_1,\dots,\mathscr{D}_n)$ in $\mathscr{P}_n{C}at^\dg_-(\k)$ is given by a dg functor $f_0\colon \mathscr{C}\to\mathscr{D}$, such that
$f_0(\mathscr{C}_i)\subset\mathscr{D}_i$, for any $i=1\dots n$.

Next, we construct a functor $\Theta\colon \Delta_f^\opp\to \Cat$, $[n]\mapsto \mathscr{P}_n\mathscr{C}at^\dg_{-}(\k)$. 
One has to define actions of the face maps $F_i\colon [n]\to [n+1]$, $1\le i\le n$, and of the degeneracy maps $D_i\colon [n+1]\to [n], 0\le i \le n$.

We set:
\begin{equation}
\begin{aligned}
\ &F_i(C;C_1,\dots, C_{n+1})=(C; C_1,\dots, C_i\cup C_{i+1}, \dots, C_{n+1})\\
&D_i(C;C_1,\dots,C_n)=(C; C_1,\dots,C_i,\emptyset,C_{i+1},\dots,C_n)
\end{aligned}
\end{equation}
where $C_{i}\cup C_{i+1}$ denotes the full subcategory of $C$ with objects $\Ob (C_i)\cup \Ob(C_{i+1})$. 

One easily checks that it indeed gives rise to a functor $\Delta_f^\opp\to\Cat$.

Now we define the category $\mathscr{PC}at^\dg_-(\k)$ as the (cofibred) {\it Grothendieck construction} of the functor $\Theta$ (see [Gr, Exp.VI], [Th]). 

Recall the general definition. 
Let $F\colon\mathscr{K}\to\Cat$ be a (strict) functor (an analogous construction also exists for $F$ a pseudo-functor).
Here we need the cofibred Grothendieck construction $\mathscr{K}{\int_c}F$, which we shortly denote by $\mathscr{K}{\int}F$. 

The objects of $\mathscr{K}\int F$ are pairs $(K,X)$ were $K\in \mathscr{K}$ and $X\in F(K)$.
A morphism $(k,x)\colon (K_1,X_1)$ $\to (K_0,X_0)$ is given by a morphism $k\colon K_1\to K_0$ in $\mathscr{K}$ and a morphism
$x\colon F(k)(X_1)\to X_0$ in $F(K_0)$.
The composition is defined as $(k,x)\ldot (k^\prime, x^\prime)=(kk^\prime, x\ldot F(k)(x^\prime))$.

The functor $p\colon \mathscr{K}\int F\to\mathscr{K}$ is defined as the projection to the first component. It makes $\mathscr{K}\int F$ a cofibred over $\mathscr{K}$ category. 

We use the following lax colimit interpretation of the Grothendieck construction:
\begin{lemma}\label{colaxlimit}
Let $\mathscr{K}$ be a small category, $F\colon \mathscr{K}\to\Cat$ a strict functor, $\mathscr{C}$ a category. Then there is a bijection between the set of functors
$g\colon \mathscr{K}\int_c F\to\mathscr{C}$, and the set of data consisting of
\begin{itemize}
\item[(1)] for each object $K\in\mathscr{K}$, a functor $g(K)\colon F(K)\to\mathscr{C}$,
\item[(2)] for each morphism $k\colon K\to K^\prime$ in $\mathscr{K}$, a natural transformation
$$
g(k)\colon g(K)\to g(K^\prime)\circ F(k)
$$
\end{itemize}
such that $g(\id_K)=\id\colon g(K)\to g(K)$, and for $K^{\prime\prime}\xleftarrow{k^\prime} K^{\prime}\xleftarrow{k} K$ one has
\begin{equation}
g(k^\prime k)=g(k^\prime)\circ g(k)
\end{equation}
\end{lemma}
See e.g. [Th, Prop. 1.3.1] for a short and direct proof.

\vspace{3mm}

We define
$$
\mathscr{PC}at_-^\dg(\k):=\Delta_f^\opp\int\Theta
$$
Explicitly, an object of $\mathscr{PC}at_-^\dg(\k)$ is given by a pair $([n], (C; C_1,\dots,C_n))$, where $C_1,\dots,C_n$ are small full subcategories of $C$, possibly empty, and a morphism 
$$
f\colon ([n], (C;C_1,\dots,C_n))\to ([m],(D; D_1,\dots, D_m))
$$
is given by a pair $f=(\sigma,F)$, where $\sigma\colon [m]\to [n]$ a morphism in $\Delta_f$, and $F\colon C\to D$ a dg functor, such that $F(C_{i})\subset D_{\sigma^*(i-1)+1}$, where $\sigma^*\colon [n-1]\to [m-1]$ is the Joyal dual to $\sigma$ element of $\Delta$. 

\begin{lemma}
The category $\mathscr{PC}at^\dg_-(\k)$ is monoidal such that  functor $p\colon \mathscr{PC}at_-^\dg(\k)\to\Delta_f^\opp$ is monoidal. 
\end{lemma}
\begin{proof}
 Let
\begin{equation}
X=([n], (\mathscr{C};\ \mathscr{C}_1,\dots,\mathscr{C}_n))\text{  and  }Y=([m], (\mathscr{D};\ \mathscr{D}_1,\dots,\mathscr{D}_m))
\end{equation}
be two objects of $\mathscr{PC}at^\dg_-(\k)$. Define their product as
\begin{equation}\label{prodnice}
X\otimes Y=\left([m+n], (\mathscr{C}\otimes\mathscr{D};\ \mathscr{C}_1\otimes\mathscr{D},\dots,\mathscr{C}_n\otimes\mathscr{D},\mathscr{C}\otimes\mathscr{D}_1,\dots,\mathscr{C}\otimes\mathscr{D}_m)\right)
\end{equation}

The monoidal product $\otimes$ on $\mathscr{PC}at^\dg_-(\k)$ extends to the morphisms in the natural way. The functor $p$ is monoidal within definition in \eqref{prodnice}.
\end{proof}

\begin{lemma}\label{sectioncart}
Any pair $(\mathscr{M}, I)$, where $\mathscr{M}$ is a $\k$-linear monoidal category, and $I\subset \mathscr{M}$ an ideal in $\mathscr{M}$, defines a monoidal section $F_\mathscr{M}\colon \Delta_f^\opp\to \mathscr{PC}at_-^\dg(\k)$ of $p\colon \mathscr{PC}at_-^\dg(\k)\to\Delta_f^\opp$.
\end{lemma}
\begin{proof}
Denote by $\mathscr{M}^{\otimes n}_i$, $1\le i\le n$, the subcategory of $\mathscr{M}^{\otimes n}$, equal to $\mathscr{M}\otimes\dots\otimes\mathscr{M}\otimes I\otimes\mathscr{M}\otimes\dots\otimes\mathscr{M}$, where the factor $I$ stands at the $i$-th place. 
To construct the section $F_\mathscr{M}$ on objects, we map the object $[n]\in \Delta_f$ to the $n$-typle 
$$
(\mathscr{M}^{\otimes n};\ \mathscr{M}^{\otimes n}_1,\dots,\mathscr{M}^{\otimes n}_n)
$$
To define it on morphisms of $\Delta_f^\opp$, we do it for $F_i^*\colon [n+1]\to [n]$ and for $D_i^*\colon [n]\to[n+1]$, dual to the corresponding morphisms of $\Delta_f$. 

On the first components $\mathscr{M}^{\otimes k}$, the morphism $F_\mathscr{M}(F_i^*)$ acts as $$\id\otimes\dots\otimes m\otimes \id\otimes\dots\otimes \id$$, where $m\colon\mathscr{M}\otimes\mathscr{M}\to \mathscr{M}$ is the monoidal product bifunctor and stands at the $i$-th position. Clearly it sends 
$\mathscr{M}^{\otimes n}_1,\dots,\mathscr{M}^{\otimes n}_{i-1}$ to $\mathscr{M}^{\otimes (n-1)}_1,\dots, \mathscr{M}^{\otimes (n-1)}_{i-1}$, correspondingly, sends $\mathscr{M}^{\otimes n}_i\cup \mathscr{M}^{\otimes n}_{i+1}$ to $\mathscr{M}^{\otimes (n-1)}_i$, and sends $\mathscr{M}^{\otimes n}_{i+2},\dots,\mathscr{M}^{\otimes n}_n$ to $\mathscr{M}^{\otimes (n-1)}_{i+1},\dots, \mathscr{M}^{\otimes (n-1)}_{n-1}$. It defines a lifting of $F_i^*$.

For the case of $D_i^*$, define the functor $F_\mathscr{M}(D_i^*)$ on the first components of the tuples as the functor $\mathscr{M}^{\otimes n}\to\mathscr{M}^{\otimes (n+1)}$, which plugs is the imbedding $e\to\mathscr{M}$ at the $i$-th position, where $e$ is the monoidal unit. It sends $\mathscr{M}^{\otimes n}_1,\dots,\mathscr{M}^{\otimes n}_{i-1}$ to $\mathscr{M}^{\otimes (n+1)}_1,\dots,\mathscr{M}^{\otimes (n+1)}_{i-1}$, $\emptyset$ to $\mathscr{M}^{\otimes (n+1)}_i$, and $\mathscr{M}^{\otimes n}_i,\dots, \mathscr{M}^{\otimes n}_n$ to $\mathscr{M}^{\otimes (n+1)}_{i+1},\dots, \mathscr{M}^{\otimes (n+1)}_{n+1}$. It defines a lifting of $D_i^*$.

The liftings of the simplicial generators respect simplicial identities, and thus the section $F_\mathscr{M}$ is defined. 
\end{proof}

Now we are going to formulate the monoidal property of the refined Drinfeld dg quotient.
\begin{prop}\label{gendr}
The functors
\begin{equation}
\begin{gathered}
\mathscr{D}r_n\colon \mathscr{P}_n{C}at^\dg_-(\k)\to\mathscr{C}at^\dg_-(\k)\\
(\mathscr{C};\ \mathscr{C}_1,\dots,\mathscr{C}_n)\mapsto\mathscr{C}/(\mathscr{C}_1,\dots,\mathscr{C}_n)
\end{gathered}
\end{equation}
give rise to a functor 
\begin{equation}
\mathscr{D}r\colon \mathscr{P}{C}at^\dg_-(\k)\to\mathscr{C}at^\dg_-(\k)
\end{equation}
The functor $\mathscr{D}r$ has a colax-monoidal structure. 
\begin{equation}
\beta_{X,Y}\colon \mathscr{D}r(X\otimes Y)\to \mathscr{D}r(X)\otimes\mathscr{D}r(Y)
\end{equation}
with $\beta_{X,Y}$ quasi-equivalences of dg categories, for all $X,Y\in\mathscr{PC}at^\dg(\k)$.
\end{prop}
\begin{proof}
First of all, we have to construct a functor $\mathscr{D}r\colon \mathscr{PC}at_-^\dg(\k)\to\Cat^\dg(\k)$. By Lemma \ref{colaxlimit},
we have to check that the functors $\mathscr{D}r_n\colon\mathscr{P}_n\mathscr{C}at_-^\dg(\k)\to\Cat^\dg(\k)$ fulfil the conditions in (2) of this Lemma.

It means that one has to define natural transformations
\begin{equation}
\mathscr{D}r(\phi)\colon \mathscr{D}r_n\Rightarrow \mathscr{D}r_m\circ \Theta(\phi)\colon \mathscr{P}_n\Cat_-^\dg(\k)\to\Cat^\dg(\k)
\end{equation}
for any morphism $\phi\colon [m]\to [n]$ in $\Delta_f$, such that for any other $\phi_1\colon [m]\to[n],\phi_2\colon [k]\to [m]$ one has
\begin{equation}\label{eqdrsimp1}
\mathscr{D}r(\phi_1\circ \phi_2)=\mathscr{D}r(\phi_2)\circ \mathscr{D}r(\phi_1)
\end{equation}
and 
\begin{equation}\label{eqdrsimp2}
\mathscr{D}r(\id)=\id
\end{equation}
We construct $\mathscr{D}r(\phi)$ when $\phi$ is a face map $F_i\colon [n]\to [n+1]$, $1\le i\le n$, or a degeneracy map $D_i\colon [n+1]\to [n]$, $0\le i\le n$. In both cases, we have to let $\mathscr{D}r(\phi)$ act on $\varepsilon_X^{i_1,\dots,i_k}$, $X\in C_{i_1}\cap\dots\cap C_{i_k}$.

One sets:
\begin{equation}
\mathscr{D}r(F_i)(\varepsilon_X^{i_1,\dots,i_k})=\begin{cases}
0&\text{ both }i,i+1\in\{i_1,\dots,i_k\}\\
\varepsilon_X^{j_1,\dots,j_k}& \text {otherwise }, j_s=i_s \text{ when }i_s\le i+1,\  j_s=i_s-1 \text{ when }i_s>i+1,\\
\end{cases}
\end{equation}
\begin{equation}
\mathscr{D}r(D_i)(\varepsilon_X^{i_1,\dots,i_k})=\varepsilon_X^{j_1,\dots,j_k}, j_s=i_s \text{ when }i_s\le i, j_s=i_s+1 \text{ when }i_s>i
\end{equation}
One checks that both $\mathscr{D}r(F_i)$ and $\mathscr{D}r(D_i)$ give rise to natural transformations (that is, are natural to dg functors $C\to C^\prime$), and that \eqref{eqdrsimp1}, \eqref{eqdrsimp2} hold for the simplicial identities.

Then $\mathscr{D}r\colon \mathscr{PC}at_-^\dg(\k)\to\Cat^\dg(\k)$ is constructed by Lemma \ref{colaxlimit}. 

\vspace{2mm}

It remains to show that $\mathscr{D}r$ is colax-monoidal. 
The colax-monoidal constraints $\beta$ can be easily constructed, based on the following remark. Let $V$ be a vector space of dimension $n$, and let $W$ be a vector space of dimension $m$.
Consider the (acyclic) complexes $\Lambda^\udot_d(V)$, $\Lambda^\udot_d(W)$, see \eqref{lambda}.
There is a natural map of complexes
\begin{equation}\label{be}
\beta_{V,W}\colon\Lambda^\udot_d(V\oplus W)\to \Lambda^\udot_d(V)\otimes\Lambda^\udot_d(W)
\end{equation}
which is a quasi-isomorphism (of acyclic complexes). In fact, $\beta_{V,W}$ is an isomorphism of the underlying graded vector spaces, compatible with the differential. 

These maps give rise to a dg functor $\beta_{X,Y}\colon \mathscr{D}r_{n+m}(X\otimes Y)\to \mathscr{D}r_n(X)\otimes\mathscr{D}r_m(Y)$, as the refined dg factor is a dg category {\it freely} generated as the envelope with the morphisms $\varepsilon_X^{i_1,\dots,i_s}$. 
Thus the left-hand side category $\mathscr{D}r_{n+m}(X\otimes Y)$ is a free envelope, which makes possible to define the map $\beta_{X,Y}$. (Remark: there does {\it not} exist any canonical map $\mathscr{D}r_n(X)\otimes\mathscr{D}r_m(Y)\to \mathscr{D}r_{n+m}(X\otimes Y)$ as the dg category $\mathscr{D}r_n(X)\otimes\mathscr{D}r_m(Y)$ is not a free envelope; the $\varepsilon$'s morphism for the different factors commute, see Remark \ref{aftermath}).

The fact that the map $\beta_{V,W}$ in \eqref{be} is a quasi-isomorphism implies that the dg functor $\beta_{X,Y}$ is a quasi-equivalence of dg categories.

\comment

Our first task is to prove that the functor $\beta_{X,Y}$ is a quasi-equivalence.
To this end, we construct a commutative diagram of dg categories:
\begin{equation}\label{psisigma}
\xymatrix{\mathscr{D}r(X\otimes Y)\ar[rr]^{\beta_{X,Y}}\ar[d]_{\Psi_{\mathscr{C}\otimes\mathscr{D}}}&&\mathscr{D}r(X)\otimes\mathscr{D}r(Y)\ar[d]^{\Psi_{\mathscr{C}}\otimes \Psi_\mathscr{D}}\\
(\mathscr{C}\otimes\mathscr{D})/(\mathscr{C}_\Sigma\otimes \mathscr{D}\cup_f \mathscr{C}\otimes\mathscr{D}_\Sigma)\ar[rr]^{L_{\mathscr{C},\mathscr{D}}}&&\mathscr{C}/\mathscr{C}_\Sigma\ \otimes\ \mathscr{D}/\mathscr{D}_\Sigma
}
\end{equation}
Here $\mathscr{C}_\Sigma$ is the full dg subcategory in $\mathscr{C}$ whose objects $\Ob\mathscr{C}_\Sigma=\Ob\mathscr{C}_1\cup\dots\cup\Ob\mathscr{C}_n$, and $\mathscr{D}_\Sigma$ is defined analogously.
The subscript $f$ in $\cup_f$ denotes ``the full subcategory generated by the union of two parts''.

The left-hand side vertical map is the map $\Psi$ constructed in \eqref{psi}, it is a quasi-equivalence by Lemma \ref{angle}; 
the right-hand side vertical map is the tensor product of the two corresponding $\Psi$'s, and therefore it is a quasi-equivalence as well. 

It remains to construct the lower horizontal map. 
Consider the map 
\begin{equation}
\mathscr{C}\otimes\mathscr{D}\to\mathscr{C}/\mathscr{C}_\Sigma\ \otimes\ \mathscr{D}/\mathscr{D}_\Sigma
\end{equation}
which is the tensor product of the canonical localization morphisms. It maps the objects in $\mathscr{C}_\Sigma\cup_f\mathscr{D}_\Sigma$ to zero objects, in the sense of Theorem \ref{druniv}. Then Theorem \ref{druniv} gives a map in the homotopy category of dg categories
\begin{equation}\label{eqdrinv}
L_{\mathscr{C},\mathscr{D}}\colon \mathscr{C}\otimes\mathscr{D}/\mathscr{C}_\Sigma\cup_f\mathscr{D}_\Sigma\to\mathscr{C}/\mathscr{C}_\Sigma\ \otimes\ \mathscr{D}/\mathscr{D}_\Sigma
\end{equation}
which is unique up to an isomorphism.

\begin{lemma}\label{lemmacorr}
The map $L_{\mathscr{C},\mathscr{D}}$ in \eqref{eqdrinv} is a quasi-equivalence of dg categories.
\end{lemma}
\begin{proof}
\end{proof}

By Lemma \ref{lemmacorr}, there exists the inverse $L^{-1}_{\mathscr{C},\mathscr{D}}$ in the homotopy category of dg categories.

For construction of this map, we use the map
\begin{equation}
\sigma\colon \Lambda^\udot_d(U_1)\otimes\Lambda^\udot_d(U_1^\prime)\to \Lambda^\udot_d(U_1^{\prime\prime})
\end{equation}
where $U_1,U_1^\prime,U_1^{\prime\prime}$ are 1-dimensional vector spaces with chosen basis vectors $e_1,e_1^\prime,e_1^{\prime\prime}$, correspondingly.

The map $\sigma$ maps $e_1\otimes e_1^\prime$ to 0, and it maps both $e_1$ and $e_1^\prime$ to $e_1^{\prime\prime}$.
The map $\sigma$ is a quasi-isomorphism of (acyclic) complexes, and the corresponding map $\Psi_2$ of dg categories is a quasi-equivalences.

Moreover, 
\begin{equation}
\xymatrix{\Lambda^\udot_d(V\oplus W)\ar[rdd]_{\Psi_{V\oplus W}}\ar[r]^{\beta_{V,W}}&\Lambda^\udot_d(V)\otimes\Lambda^\udot(W)\ar[d]^{\Psi_V\otimes\Psi_W}\\
&\Lambda^\udot_d(U_1)\otimes\Lambda^\udot_d(U_1^\prime)\ar[d]^{\sigma}\\
&\Lambda^\udot_d(U_1^{\prime\prime})}
\end{equation}
where the arrows are defined as in \eqref{be}, is commutative. Therefore, the diagram \eqref{psisigma} is commutative.

As its two angle maps are quasi-equivalences, it follows from the 2-out-of-3 property of quasi-equivalences that the map $\beta$ is a quasi-equivalence.

\endcomment

It remains to prove that $\beta$ defines a colax-monoidal structure. It can be easily seen using the complexes $\Lambda^\udot_d(V)$. Namely, for three vector spaces with chosen bases $V,W,Z$, the diagram
\begin{equation}
\xymatrix{
\Lambda_d^\udot(V\oplus W\oplus Z)\ar[rr]\ar[d]&&\Lambda_d^\udot(V\oplus W)\otimes\Lambda^\udot_d(Z)\ar[d]\\
\Lambda_d^\udot(V)\otimes\Lambda^\udot_d(W\oplus Z)\ar[rr]&&\Lambda^\udot_d(V)\otimes\Lambda^\udot_d(W)\otimes\Lambda^\udot_d(Z)
}
\end{equation}
where the arrows are defined as in \eqref{be}, commutes. It implies the corresponding colax-monoidal property for $\beta_{X,Y}$.
\end{proof}

\begin{remark}\label{aftermath}{\rm
Although the map $\beta_{V,W}\colon\Lambda_d^\udot(V\oplus W)\to \Lambda_d^\udot(V)\otimes\Lambda_d^\udot(W)$ is a quasi-isomorphism (of acyclic complexes),
the refined Drinfeld dg quotient functor admit only a colax-monoidal but not a lax-monoidal structure.
The matter is that $\Dr(X\otimes Y)$ is a {\it free} envelope of the morphisms in $X\otimes Y$ with the elements in $\Lambda_d^\udot(V\oplus W)$ assigned to the corresponding objects. Contrary, $\Dr(X)\otimes \Dr(Y)$ is {\it not} a free envelope, namely the elements of $\Lambda^\udot(V)$ and $\Lambda^\udot(W)$ do commute in $\Dr(X)\otimes\Dr(Y)$. Therefore, to construct a map $\Dr(X)\otimes\Dr(Y)\to\Dr(X\otimes Y)$ one needs to fulfil these relations which hold in the source category. A helpful analogy: for a free dg associative algebra $A$, to define an algebra map $A\to B$ to another (in general not free) dg associative algebra $B$, it is enough to define this map on the generators of $A$, in the way compatible with the action of differential.
}
\end{remark}

\section{\sc Deligne conjecture for essentially small abelian monoidal categories}\label{deln=1}
Here we prove the Deligne conjecture for 1-monoidal categories in the case when the abelian category $\mathscr{A}$ is essentially small. The case of the $n$-fold monoidal abelian categories is treated in Section \ref{deln}. The assumption that $\mathscr{A}$ is essentially small is not really necessary and can be weaken. We decided to treat this case firstly to make a more clear exposition of the main ideas. In this case we have not any set-theoretical troubles in use of the dg quotient, what has many technical advantages. The general case, when the triangulated category $H^0\mathscr{A}^\dg$ is generated by a {\it set} of perfect objects, will be considered in our subsequent papers.

The case when $\mathscr{A}$ is essentially small covers the case of $\mathscr{U}$-generated $C$-modules, where $C$ is an algebra whose underlying set is a $\mathscr{U}$-set, for a universe $\mathscr{U}$. In particular, it covers the ``classical'' example of the category of $A$-bimodules, as well as the category of left modules over a bialgebra $B$. It covers many other examples of algebraic origin.

\subsection{\sc Weak compatibility between the exact and the monoidal structure in $\mathscr{A}$}
\begin{defn}\label{defn41}{\rm
Let $\mathscr{A}$ be an abelian category, with a monoidal structure $(\otimes,e)$ on it.
Denote by $\mathscr{A}^\dg$ the dg category of bounded from above complexes in $\mathscr{A}$, and let $\mathscr{I}\subset \mathscr{A}^\dg$ be the full dg subcategory of acyclic objects. Consider the full additive subcategory $\mathscr{A}_0\subset\mathscr{A}$ where $X\in\mathscr{A}_0$ iff
$X\otimes I$ and $I\otimes X$ are acyclic, for any acyclic $I\in\mathscr{I}$. Denote by $\mathscr{A}_0^\dg\subset\mathscr{A}^\dg$ the full dg subcategory of bounded from above complexes in $\mathscr{A}_0$. We say that the exact and the monoidal structures on $\mathscr{A}$ are {\it weakly compatible} if the natural embedding
$$
\mathscr{A}_0^\dg\hookrightarrow \mathscr{A}^\dg
$$
is a quasi-equivalence of dg categories.
}
\end{defn}

\begin{lemma}\label{lemmaklemma}
Let $\mathscr{A}$ be an abelian monoidal category, whose abelian and monoidal structures are weakly compatible, see Definition above.
Let $\mathscr{I}\subset \mathscr{A}^\dg$ be the full dg subcategory of acyclic objects, and $\mathscr{I}_0=\mathscr{I}\cap\mathscr{A}_0^\dg$.
Then the natural embedding $\mathscr{A}_0^\dg\hookrightarrow\mathscr{A}^\dg$ (which is a quasi-equivalence) induces a quasi-equivalence
$\mathscr{I}_0\hookrightarrow\mathscr{I}$.
\end{lemma}
\begin{proof}
It is enough to show that the induced map $H^0\mathscr{I}_0\to H^0\mathscr{I}$ of homotopy categories is essentially surjective.
Let $X\in H^0\mathscr{I}$. We know by the assumption that the map $H^0\mathscr{A}_0^\dg\to H^0\mathscr{A}^\dg$ is essentially surjective.
In particular, $X$ is isomorphic in $H^0\mathscr{A}^\dg$ to some object $Y\in H^0\mathscr{A}_0^\dg$. The isomorphic objects in the homotopy category have isomorphic cohomology, and $X$ is acyclic. Therefore, $Y$ is acyclic as well, and thus $Y\in\mathscr{I}_0$.
\end{proof}

\begin{lemma}
Let $\k$ be any field, and let $\mathscr{A}$ be an abelian $\k$-linear category with a monoidal structure $(\otimes, e)$.
Suppose that $\mathscr{A}$ has enough projective objects, and $\otimes\colon \mathscr{A}\times\mathscr{A}\to\mathscr{V}ect(\k)$ is a right exact bifunctor. Then the the exact structure of $\mathscr{A}$ and the monoidal structure of $\mathscr{A}$ are weakly compatible.
\end{lemma}

\qed

\subsection{\sc Deligne conjecture}\label{proofne1}
Here we prove:
\begin{theorem}\label{delignesimple}
Let $\mathscr{A}$ be an abelian $\k$-linear category, with a monoidal structure $(\otimes,e)$. Suppose the exact and the monoidal structure on $\mathscr{A}$ are weakly compatible, and $\mathscr{A}$ is essentially small. Then the graded $\k$-vector space $\RHom^\udot_\mathscr{A}(e,e)$ has a natural structure of a Leinster 2-algebra over $\k$, whose second product is homotopy equivalent to the Yoneda product.
\end{theorem}

We start with a general Lemma:
\begin{lemma}\label{lemmaxxx}
Let $\mathscr{C},\mathscr{D}$ be two $\k$-linear dg categories, $F\colon \mathscr{C}\to\mathscr{D}$ a dg functor which is a quasi-equivalence.
Let $\mathscr{C}_0\subset \mathscr{C}$, $\mathscr{D}_0\subset\mathscr{D}$ be two essentially small full dg subcategories, such that $F$ restricts to a quasi-equivalence $F_0\colon \mathscr{C}_0\to\mathscr{D}_0$. Then the dg functor $\overline{F}\colon \mathscr{C}/\mathscr{C}_0\to\mathscr{D}/\mathscr{D}_0$ of the Drinfeld's dg quotients, is a quasi-equivalence.
\end{lemma}
\begin{proof}
Let $X,Y$ be any two objects in $\mathscr{C}$. We firstly prove that the map of $\Hom$-complexes
\begin{equation}
\overline{F}(X,Y)\colon \Hom_{\mathscr{C}/\mathscr{C}_0}(X,Y)\to\Hom_{\mathscr{D}/\mathscr{D}_0}(FX,FY)
\end{equation}
is a quasi-isomorphism of complexes.

To this end, consider the cone $\Cone(\overline{F}(X,Y))$; one needs to prove that it is an acyclic complex.

The complexes $\Hom_{\mathscr{C}/\mathscr{C}_0}(X,Y)$, $\Hom_{\mathscr{D}/\mathscr{D}_0}(FX,FY)$ admit natural ascending filtrations, see \eqref{drinmor}. The map $\overline{F}(X,Y)$ is a map of filtered complexes. Therefore, the cone $\Cone(\overline{F}(X,Y))$ inherits this filtration. We compute the corresponding to this filtration spectral sequence.

This spectral sequence $\{E_r^{p,q},d_r\}$ is non-zero in I and in IV quarters, $d_r\colon E_r^{p,q}\to E_r^{p-r,q+r+1}$. Therefore, it converges to
$\gr H^\udot(\Cone(\overline{F}(X,Y)))$ by dimensional reasons.

The term $E_0=\oplus_p F_{p+1}/F_p$ is easy to describe. Namely, it is a free expression \eqref{drinmor}, with zero values of the differential on objects $k[1]$. The cone of the map of complexes $F(Y_i,Y_{i+1})\colon\mathscr{C}(Y_i,Y_{i+1})\to\mathscr{D}(FY_i,FY_{i+1})$ is acyclic, as $F$ is a quasi-equivalence. Therefore, the term $E_1$ vanishes everywhere, and the map $\overline{F}(X,Y)$ is a quasi-isomorphism.

It remains to prove that the corresponding map of the homotopy categories,
\begin{equation}
H^0\overline{F}\colon H^0(\mathscr{C}/\mathscr{C}_0)\to H^0(\mathscr{D}/\mathscr{D}_0)
\end{equation}
is an equivalence of categories. It enough to prove $\overline{F}$ is essentially surjective, which is clear.
\end{proof}

Now we prove the Theorem.

\begin{proof}
Recall the category $\mathscr{PC}at_-^\dg(\k)$ introduced in Section \ref{section33}. It is a monoidal category.

Let $\mathscr{A}$ be as in the statement of Theorem. In notations of Definition \ref{defn41}, denote $\mathscr{J}_0=\mathscr{A}_0\cap\mathscr{J}$.
We construct, out of $\mathscr{A}$, a Leinster monoid $F_\mathscr{A}$ in $\mathscr{PC}at^\dg_{-}(\k)$, as in Lemma \ref{sectioncart}.
Its underlying functor $F_{\mathscr{A}}\colon \Delta_f^\opp\to \mathscr{PC}at^\dg_{-}(\k)$ is defined on objects as
\begin{equation}\label{who1}
F_{\mathscr{A}}([n])=(\mathscr{A}_0^{\otimes n};\ \mathscr{C}^{[n]}_1,\dots,\mathscr{C}^{[n]}_n)
\end{equation}
where
\begin{equation}\label{who2}
\begin{aligned}
\ &\mathscr{C}^{[n]}_1=\mathscr{J}_0\otimes\mathscr{A}_0\otimes\mathscr{A}_0\otimes\dots\otimes \mathscr{A}_0\ \ \text{(totally $n-1$ factors $\mathscr{A}_0$)}\\
&\mathscr{C}^{[n]}_2=\mathscr{A}_0\otimes\mathscr{J}_0\otimes \mathscr{A}_0\otimes\dots\otimes \mathscr{A}_0\ \ \text{(totally $n-1$ factors $\mathscr{A}_0$)}\\
&\dots\dots\dots\\
&\mathscr{C}^{[n]}_n=\mathscr{A}_0\otimes\mathscr{A}_0\otimes\dots\otimes\mathscr{A}_0\otimes\mathscr{J}_0\ \ \text{(totally $n-1$ factors $\mathscr{A}_0$)}
\end{aligned}
\end{equation}
(for the dg category $\mathscr{C}^{[n]}_i$, the factor $\mathscr{J}_0$ is at $i$-th place).

The claim that $F_\mathscr{A}\colon \Delta_f^{\opp}\to\mathscr{PC}at^\dg_-(\k)$ is a functor, follows from the fact the $\mathscr{J}_0$ is a two-sided ideal in $\mathscr{A}_0$: $\mathscr{A}_0\odot\mathscr{J}_0\subset \mathscr{J}_0$, $\mathscr{J}_0\odot\mathscr{A}_0\subset \mathscr{J}_0$, which follows from Definition \ref{defn41}. Then we deduce that the Leinster monoid defined out of a strict monoidal category $\mathscr{A}_0$ as in Definition \ref{leinster1}, descents to a Leinster monoid structure on $F_{\mathscr{A}}$ in $\mathscr{PC}at^\dg_-(\k)$, with the identity colax-maps $F_{\mathscr{A}}([m+n])\to F_\mathscr{A}([m])\otimes F_\mathscr{A}([n])$.

Next, we apply the refined Drinfeld dg quotient functor $\mathscr{D}r\colon \mathscr{PC}at_-^\dg(\k)\to\mathscr{C}at^\dg_-(\k)$, defined in Proposition \ref{gendr}. By this Proposition, the functor $\mathscr{D}r$ has a natural colax-monoidal structure, whose colax maps are quasi-equivalences of dg categories.

We have the composition of colax-monoidal functors:
\begin{equation}
\Delta_f^\opp\xrightarrow{F_\mathscr{A}}\mathscr{PC}at^\dg_-(\k)\xrightarrow{\mathscr{D}r}\mathscr{C}at^\dg_-(\k)
\end{equation}
As a composition of such ones, the total functor is a colax-monoidal functor, whose colax maps are quasi-equivalences.

Next, me notice that for each $[n]$, the category $\mathscr{D}r\circ F_\mathscr{A} [n]$ has a distinguished object $*_{[n]}$. In fact,
\begin{equation}\label{comp1}
*_{[n]}=\mathscr{D}r(e\otimes e\otimes\dots \otimes e) \ \ \text{$n$ factors $e$}
\end{equation}
where $e$ is the unit in $\mathscr{A}$ (which belongs to $\mathscr{A}_0$ and is the unit in it). It means that the image of composition \eqref{comp1}
takes values in the category $\mathscr{C}at^\dg_{-*}(\k)$ of the corresponding dg categories {\it with a marked object}.

There is a natural functor
\begin{equation}\label{rhom}
\mathscr{H}om\colon\mathscr{C}at^\dg_{-*}(\k)\to\mathscr{M}on(\k),\ \ \mathscr{C}\mapsto \Hom_\mathscr{C}(*,*)
\end{equation}
to the category of monoids in $\mathscr{V}ect(\k)$ (that is, to the category of associative dg algebras over $\k$).

The functor $\mathscr{H}om$ is colax-monoidal with the identity colax maps.

Now the composition of \eqref{comp1} with the functor $\mathscr{H}om$ gives a functor
\begin{equation}\label{comp2}
\mathscr{H}om\circ \mathscr{D}r\circ F_\mathscr{A}\ \colon\ \Delta^\opp_0\to\mathscr{M}on(\k)
\end{equation}
As a composition of such ones, it is a colax-monoidal functor, whose colax maps are quasi-isomorphisms of dg algebras over $\k$.

\begin{klemma}\label{kld}
$\mathscr{H}om\circ \mathscr{D}r\circ F_\mathscr{A}([1])=\RHom^\udot_{\mathscr{A}}(e,e)$
\end{klemma}

\begin{proof}
The composition in the statement of Key-Lemma is the $\Hom$-complex $(\mathscr{A}_0/\mathscr{J}_0)(e,e)$.
By Lemmas \ref{lemmaxxx} and \ref{lemmaklemma}, we know that $\mathscr{A}_0/\mathscr{J}_0$ is quasi-equivalent to $\mathscr{A}/\mathscr{J}$.
Now the claim follows from Proposition \ref{keller}, which gives the Keller description of the dg quotient for a pre-triangulated dg category $\mathscr{C}$ and its full pre-triangulated dg subcategory $\mathscr{N}$.
\end{proof}

\begin{remark}{\rm
Note that we have not such explicit descriptions for $\mathscr{H}om\circ \mathscr{D}r\circ F_\mathscr{A}([\ell])$
for $\ell>1$. Indeed, the tensor product of two pre-triangulated categories is not pre-triangulated, in general. 
Consequently, the Keller description of dg quotent given in Proposition \ref{keller} can not be applied.
That is, the higher components in the constructed Leinster monoid in $\Alg(\k)$, are given rather indirectly. 
}
\end{remark}

What we get is the following. We have constructed a Leinster monoid $\mathscr{F}_\mathscr{A}$ in the monoidal category $\Alg(\k)$ of dg algebras over $\k$,
whose first component $\mathscr{F}_\mathscr{A}[1]$ is quasi-isomorphic to $\RHom^\udot_\mathscr{A}(e,e)$. The higher components $\mathscr{F}_{\mathscr{A}}[n]$, $n>1$, are hard to compute, but we do not need to know them explicitly. Only what we need is that for {\it some} higher components, the dg algebra $\RHom^\udot_\mathscr{A}(e,e)$ can be complemented (as the first component) to a Leinster monoid in the monoidal category $\Alg(\k)$ of dg algebras over $\k$. It follows, by Definition \ref{leinstern}, that $\RHom^\udot_{\mathscr{A}}(e,e)$ is a Leinster 2-algebra in $\Vect(\k)$.
\end{proof}

\section{\sc The polymonoidal (op)lax-functor defined by an $n$-fold monoidal category}\label{allthat}
\subsection{}
Before starting to deal with the Deligne conjecture for $n$-fold monoidal [BFSV] abelian categories for $n>1$, we need to recall some definitions on (monoidal) (op)lax-functors. Here we explain why. 

Given an $n$-fold monoidal category $\k$-linear category $\mathscr{C}$, $n>1$, we want to encode this data in a colax-monoidal functor 
$F=F_{\mathscr{C}}\colon(\Delta_f^\opp)^{\times n}\to \Cat$, as we have done for the case $n=1$. 

However, for $n>1$ it does not work immediately: when we define $F$ by
$$
F([m_1],[m_2],\dots,[m_n])=\mathscr{C}^{\otimes (m_1\cdot m_2\cdot ...\cdot m_n)}
$$
the morphisms from different factors $\Delta_f^\opp$ {\it do not commute}, but are expressed through the Eckmann-Hilton maps $\eta_{ij}$, though they do commute in the source category. This point was emphasized in [BFSV, Theorem 2.1]. What we get is not a honest functor but a lax-functor (see Definition \ref{quasi} below), as is proven in loc.cit.

[BFSV] deals with the case of set-enriched $n$-fold monoidal categories, where such a category gives rise to a lax-functor $(\Delta^\opp)^n\to \Cat$, see loc.cit., Theorem 2.1. Out Theorem \ref{monquasi} is a substitute for loc.cit. for the non cartesian-monoidal case.

We need to define what a (poly)monoidal (op)lax-functor is. We are lucky that our problem permits us to restrict ourselves 
with {\it strict} morphisms of (op)lax-functors, {\it strict} morphisms of (strict) (op)lax-bifunctors, etc. Otherwise, we ought to deal with ``higher coherence conditions'', see e.g. [DS1,2].

\subsection{\sc Definitions}
Recall that the difference between (strict) 2-categories and bicategories (the latter is more general than the former) is that in 2-categories the composition of 1-arrows is strictly associative, whence in bicategories it is associative up to 2-arrows (which 
are invertible and fulfil some coherence, see [ML, XII]). Any bi-category is bi-equivalent to a 2-category, see e.g. [Le2, Sect. 2.3].
This result can be considered as a generalization of the MacLane coherence theorem [ML, XI.3], as a strict (corresp., with relaxed up to a coherent isomorphism associativity) monoidal category gives rise to a 2-category (corresp., to a bicategory) with a single object. 
\begin{remark}\label{remgps}{\rm
The coherence theorem for {\it monoidal bi-categories} in its naive form fails, see [GPS], [DS1,2], [Sim]. It is not true in general that a suitably defined lax-monoid in bi-categories [DS2] is equivalent to a strict monoid in 2-categories with its cartesian monoidal product. There is another monoidal product on 2-categories called the Gray product, which is a relaxed version of the cartesian one. The correct coherence theorem [GPS] says that any lax-monoid in bi-categories is equivalent to a strict monoid in 2-categories, but with its Gray product. See also Remark \ref{remgpsbis} below.
}
\end{remark}

\begin{defn}\label{quasi}{\rm
Let $\mathscr{K}$ be a category and $\mathscr{X}$ a 2-category. A {\it lax-functor} $F\colon \mathscr{K}\to\mathscr{X}$ consists of functions assigning:
\begin{itemize}
\item[(1)] to each object $k\in \mathscr{K}$ an object $F(k)\in\mathscr{X}$;
\item[(2)] to each morphism $t\colon k_1\to k_2$ in $\mathscr{K}$, a 1-arrow $F(t)\colon F(k_1)\to F(k_2)$ in $\mathscr{X}$;
\item[(3)] to each composible pair of morphisms $k_1\xrightarrow{t_1}k_2\xrightarrow{t_2}k_3$ in $\mathscr{K}$, a 2-arrow
$f(t_1,t_2)\colon F(t_2)\circ F(t_1)\Rightarrow F(t_2\circ t_1)$ in $\mathscr{X}$;
\item[(4)] to each object $k\in \mathscr{K}$ a 2-arrow $f(k)\colon \id_{F(k)}\Rightarrow F(\id_k)$
\end{itemize}
They must satisfy the following conditions:
\begin{itemize}
\item[(i)] for any three composable morphisms $k_1\xrightarrow{t_1}k_2\xrightarrow{t_2}k_3\xrightarrow{t_3}k_4$
in $\mathscr{K}$ one has:
\begin{equation}
f(t_2t_1,t_3)\circ (F(t_3)\circ f(t_1,t_2))=f(t_1,t_3t_2)\circ (f(t_2,t_3)\circ F(t_1))
\end{equation}
\item[(ii)] for any morphism $t\colon k_1\to k_2$ in $\mathscr{K}$ one has
\begin{equation}
f(1,t)\circ (f(k_0)\circ F(k_1))=\id_{F(t)}=f(t,1)\circ (F(k_0)\circ f(k_1)) 
\end{equation}
\end{itemize}

}
\end{defn}
We will need as well the dual concept:
\begin{defn}\label{quasibis}{\rm
Let $\mathscr{K}$ be a category and $\mathscr{X}$ a 2-category. An {\it oplax-functor} $F\colon \mathscr{K}\to\mathscr{X}$ consists of functions assigning:
\begin{itemize}
\item[(1)] to each object $k\in \mathscr{K}$ an object $F(k)\in\mathscr{X}$;
\item[(2)] to each morphism $t\colon k_1\to k_2$ in $\mathscr{K}$, a 1-arrow $F(t)\colon F(k_1)\to F(k_2)$ in $\mathscr{X}$;
\item[(3)] to each composible pair of morphisms $k_1\xrightarrow{t_1}k_2\xrightarrow{t_2}k_3$ in $\mathscr{K}$, a 2-arrow
$g(t_1,t_2)\colon F(t_2\circ t_1)\Rightarrow  F(t_2)\circ F(t_1)$ in $\mathscr{X}$;
\item[(4)] to each object $k\in K$ a 2-arrow $g(k)\colon  F(\id_k)\Rightarrow\id_{F(k)}$
\end{itemize}
They must satisfy the conditions dual to those in the definition of a lax-functor, see Definition \ref{quasi}, (i),(ii). 
}
\end{defn}
In this paper, we consider only the case when $\mathscr{X}=\Cat$ or $\mathscr{X}=\Cat^\dg(\k)$. 
\begin{remark}{\rm
The 2-arrows $f(t_1,t_2)$ and $f(k)$ are not assumed to be invertible. When all they are invertible, a lax functor is called a pseudo-functor. In this case a lax-functor defines an oplax-functor. As well an oplax-functor with invertible $g(t_1,t_2)$ and $g(k)$ defines a lax-functor. 
}
\end{remark}

There is a concept of a lax-transform, which extends the concept of natural transformation between honest functors, to the case of lax-functors. 

\begin{defn}\label{lont}{\rm
\begin{itemize}
\item[(1)]
Let $F,G\colon K\to\mathscr{X}$ be two lax-functors to a 2-category $\mathscr{X}$. A (left) {\it lax-transform} $\varepsilon\colon F\Rightarrow G$ consists of functions assigning
\begin{itemize}
\item[(i)] to each object $k\in K$ a 1-arrow $\varepsilon(k)\colon F(k)\to G(k)$;
\item[(ii)] to each morphism $t\colon k_1\to k_2$ in $K$ a 2-arrow $\varepsilon(t): G(t)\circ \varepsilon(k_1)\Rightarrow \varepsilon(k_2)\circ F(t)$
\end{itemize}
subject to a list of axioms which the reader can find e.g. in [Th, Def. 3.1.3]. 
A lax-transform is called {\it a 2-isomorphism} if $\varepsilon(t)$ is an isomorphism 2-arrow for any $t$;
\item[(2)] a lax-transform $\varepsilon\colon F\Rightarrow G$ between two (op)lax-functors is called {\it strict} if all 2-arrows $\varepsilon(t)$, $t$ a morphism in $K$, are the identity arrows. 
\end{itemize}
}
\end{defn}
\begin{remark}{\rm 
The concept defined in Definition \ref{lont} may be also called a {\it left} lax-transform. Similarly, one can define a right lax-transform. In the case when a (left) lax-transform is an isomorphism it can be as well regarded as a right lax-transform. 
}
\end{remark}

\begin{defn}\label{bif}{\rm 
Let $\mathscr{K},\mathscr{L}$ be ordinary categories, $\mathscr{C}$ a 2-category. By a {\it strict} (op)lax-bifunctor $F\colon \mathscr{K}\times \mathscr{L}\to\mathscr{C}$ we mean an assignment defining an (op)lax functor of each argument for fixed other argument, such that for $f\colon X\to X^\prime$ a morphism in $\mathscr{K}$, $g\colon Y\to Y^\prime$ a morphism in $\mathscr{L}$, one has the strict commutativity:
\begin{equation}\label{eqbif}
F(\id_{X^\prime}\times g)\circ F(f\times\id_{Y})= F(f\times \id_{Y^\prime}) \circ F(\id_X\times g)
\end{equation}
We denote by $F(f\times g)$ the equal expressions in  \eqref{eqbif}.
}
\end{defn}
In this paper, we deal only with {\it strict} morphisms of (op)lax (bi-,poly-)functors, and the definition below restricts by this case. 
\begin{defn}\label{bifmap}
{\rm
Let $\mathscr{K},\mathscr{L}$ be ordinary categories, and $\mathscr{C}$ be a 2-category. Let $F_1,F_2\colon \mathscr{K}\times\mathscr{L}\to\mathscr{C}$ be two strict (op)lax bi-functors, see Definition \ref{bif}.
A {\it strict morphism of strict (op)lax bi-functors} $\Psi\colon F_1\Rightarrow F_2$ assigns to each objects $(k,\ell)$ of $\mathscr{K}\times \mathscr{L}$ a 1-arrow $$\Psi(k,\ell)\colon 
F_1(k,\ell)\to F_2(k,\ell)$$
in $\mathscr{C}$, such that for any morphism $f\colon k\to k^\prime$ and $g\colon\ell\to\ell^\prime$ the diagram below {\it strictly} commutes:
\begin{equation}\label{bifmapeq}
\xymatrix{
F_1(k,\ell)\ar[rr]^{\Psi(k,\ell)}\ar[d]_{F_1(f\times g)}&&F_2(k,\ell)\ar[d]^{F_2(f\times g)}\\
F_1(k^\prime,\ell^\prime)\ar[rr]^{\Psi(k^\prime,\ell^\prime)}&&F_2(k^\prime,\ell^\prime)
}
\end{equation}
}
\end{defn}
\begin{remark}\label{remgpsbis}{\rm
In the following definition we assume that $\mathscr{C}$ is a {\it strict} monoidal 2-category (with the product denoted $\otimes$) , with respect to the cartesian monoidal structure $\times$ on the category of all 2-categories. That is, we have a {\it strict} bifunctor $\otimes: \mathscr{C}\times \mathscr{C}\to \mathscr{C}$, where the strictness means that the morphisms acting on different factors strictly commute (which is granted by considering the cartesian product on the category of 2-categories), and this bifunctor is {\it strictly} associative. In particular, for an ordinary category $\mathscr{K}$ and an (op)lax-functor $F\colon \mathscr{K}\to\mathscr{C}$, the (op)lax-bifunctor functor $F_2\colon \mathscr{K}\times\mathscr{K}\to\mathscr{C}$, defined on objects as $F_2(X\times Y)=F_2(X)\otimes F_2(Y)$, is a strict (op)lax-bifunctor, see Definition \ref{bif}. This bifunctor $F_2$ is obtained as the composition $$\mathscr{K}\times \mathscr{K}\xrightarrow{F\times F}\mathscr{C}\times\mathscr{C}\xrightarrow{\otimes}\mathscr{C}$$

The 2-categories $\Cat$ and $\Cat^\dg(\k)$ are monoidal with the monoidal with respect to the cartesian product on $2-\Cat$, but they are not, however, strictly associative (they are associative up to a coherent system of isomorphisms). In general, it is {\it not} true that such a category is 2-equivalent to a strict monoidal category with respect to the cartesian product on $2-\Cat$, see Remark \ref{remgps}. The 2-categories $\Cat$ and $\Cat^\dg(\k)$ form a lucky exception; they are 2-equivalent by strictly associative categories monoidal 2-categories with respect to the cartesian product on $2-\Cat$. It can be provided by an explicit construction mimicking the MacLane construction in his proof of coherence theorem for monoidal 1-categories (which holds in general), see [ML, Section XI.3]. There is a relaxed monoidal product on $2-\Cat$, the Gray product, for which the strictification of the associativity is true in general. The price one pays for that is considering the bifunctors for which \eqref{eqbif} fails to hold on the nose, but holds up to a 2-arrow, which are subject to some coherence, and so on.

The fact that $\Cat$ and $\Cat^\dg(\k)$ are equivalent to strictly associative 2-categories with respect to the cartesian product, makes it possible to work with them as if they were strictly associative 2-categories with respect to the cartesian product on $2-\Cat$. By this reason, we can ignore the issue with non-associativity of the monoidal product in $\Cat$ and $\Cat^\dg(\k)$.}
\end{remark}

\begin{defn}\label{bifmon}{\rm
Let $\mathscr{K}$ be a strict monoidal 1-category, and $\mathscr{C}$ a strict monoidal 2-category (with respect to the cartesian product on $2-\Cat$).
An (op)lax-functor $F\colon\mathscr{K}\to \mathscr{C}$ is called {\it a strict  monoidal (op)lax-functor} if there is a strict map of (op)lax-bifunctors (see Definition \ref{bifmap}), 
$$
\Theta\colon F_1\Rightarrow F_2\colon \mathscr{K}\times\mathscr{K}\to\mathscr{C}
$$
(where $F_1$ is trivially a strict (op)lax-bifunctor, for $F_2$ see the discussion just above), 
and a map
$$
\eta\colon F(1_{\mathscr{K}})\to 1_{\mathscr{C}}
$$
which makes the following diagrams commute:
\begin{equation}
\xymatrix{
F(X\otimes Y\otimes Z)\ar[rr]^{\Theta(X,Y\otimes Z)}\ar[d]_{\Theta(X\otimes Y, Z)}&&F(X)\otimes F(Y\otimes Z)\ar[d]^{\id\otimes\Theta(Y,Z)}\\
F(X\otimes Y)\otimes F(Z)\ar[rr]^{\Theta(X,Y)\otimes\id}&&F(X)\otimes F(Y)\otimes F(Z)
}
\end{equation} 
and
\begin{equation}
\begin{array}{cc}
\xymatrix{
F(X)\ar[d]_{\id}\ar[r]^{\id}& F(X\otimes 1_{\mathscr{K}})\ar[d]^{\Theta}\\
F(X)\otimes 1_\mathscr{C}& F(X)\otimes F(1_\mathscr{K})\ar[l]_{\id\otimes \eta}
}
&
\xymatrix{
F(X)\ar[d]_{\id}\ar[r]^{\id}& F(1_{\mathscr{K}}\otimes X)\ar[d]^{\Theta}\\
 1_\mathscr{C}\otimes F(X)& F(1_\mathscr{K})\otimes F(X)\ar[l]_{\eta\otimes\id}
}
\end{array}
\end{equation}
}
\end{defn}

The last definition in this series specifies what a strict poly-monoidal (op)lax-functor is.
\begin{defn}\label{polymon}{\rm
Let $\mathscr{K}_1,\dots,\mathscr{K}_n$ be strict monoidal 1-categories, and let $\mathscr{C}$ be a strict monoidal 2-category (with respect to the cartesian product on $2-\Cat$). An (op)lax-functor 
$$F\colon \mathscr{K}=\mathscr{K}_1\times\dots\times\mathscr{K}_n\to\mathscr{C}$$
is called {\it a strict polymonoidal (op)lax-functor} if:
\begin{itemize}
\item[(1)] for any $1\le i\le n$ there is an (op)lax morphism of the (op)lax functors $$\Theta_i\colon F_1^i\Rightarrow F_2^i$$ where
$$
F_1^i,F_2^i\colon (\mathscr{K}_1\times\dots\times {\mathscr{K}}_i\times\times \mathscr{K}_i\dots\times \mathscr{K}_n)\to \mathscr{C}
$$
$$
F_1^i(k_1,\dots,{k}_i^\prime,k_i^\dprime\dots,k_n)=
F(k_1,\dots,k_{i-1},k_i^\prime\otimes k_i^{\prime\prime},\dots,k_n)
$$
$$
F_2^i(k_1,\dots,{k}^\prime_i,k_i^\dprime,\dots,k_n)=
F(k_1,\dots,k_i^\prime,\dots,k_n)\otimes F(k_1,\dots,k_i^{\prime\prime},\dots,k_n)
$$
and $F_1^i,F_2^i$ are defined on the morphisms accordingly,
\item[(2)] 
for any $1\le i\le n$ and any $\{k_j\in\mathscr{K}_j\}_{j\ne i}$ a morphism
\begin{equation}
U_i(k_1,\dots,k_n)\colon F(k_1,\dots,k_{i-1},e_i,k_{i+1},\dots,k_n)\to e_{\mathscr{C}}
\end{equation}
where $e_i$ is the monoidal unit in $\mathscr{K}_i$, and $e_\mathscr{C}$ is the monoidal unit in
$\mathscr{C}$.
\end{itemize}

These maps $\Theta_i$ should fulfil the following conditions:
\begin{itemize}
\item[(i)] for any $1\le i\le n$, and any morphisms $\alpha_j$ in $\mathscr{K}_j$, $j\ne i$, and $\alpha^\prime_i,\alpha_i^\dprime$ in $\mathscr{K}_i$, the diagram below {\it strictly commutes}:
\begin{equation}\label{polyeq}
\xymatrix{
F_1(k_1,\dots, k_i^\prime,k_i^\dprime,\dots,k_n)\ar[rr]^{\Theta}\ar[d]_{(f^\otimes)_*}&&F_2(k_1,\dots, k_i^\prime,k_i^\dprime,\dots,k_n)\ar[d]^{(f^\prime)_*\otimes (f^\dprime)_*}\\
F_1(\ell_1,\dots,\ell_i^\prime,\ell_i^\dprime,\dots,\ell_n)\ar[rr]^{\Theta}&&F_2(\ell_1,\dots,\ell_i^\prime,\ell_i^\dprime,\dots,\ell_n)
}
\end{equation}
where
$$
f^{\otimes}=(\alpha_1,\dots,\alpha_i^\prime\otimes\alpha_i^\dprime,\dots,\alpha_n),\ \ f^\prime=
(\alpha_1,\dots,\alpha_i^\prime,\dots,\alpha_n),\ \ f^\dprime=(\alpha_1,\dots,\alpha_i^\dprime,\dots,\alpha_n)
$$
\item[(ii)] $U_i$ agrees with the morphisms in $\mathscr{K}_1\times\dots\times\hat{\mathscr{K}}_i\times\dots\times\mathscr{K}_n$,
\item[(iii)] for any $i$ and for any $\{k_j\in\mathscr{K}_j\}_{1\le j\le n}$ the composition
\begin{equation}
F(k_1,\dots,k_n)\xrightarrow{\Theta_i}F(k_1,\dots,k_n)\otimes F(k_1,\dots,e_i,\dots,k_n)\xrightarrow{\id\otimes U_i}F(k_1,\dots,k_n)
\end{equation}
is the identity map,
\item[(iv)]
for any $i_1<i_2$ and for any $\{k_j\in\mathscr{K}_j\}_{j\ne i_1,i_2}$ one has
$$
U_{i_1}(k_1,\dots,\hat{k}_{i_1},\dots,\hat{k}_{i_2},\dots,k_n)=U_{i_2}(k_1,\dots,\hat{k}_{i_1},\dots,\hat{k}_{i_2},\dots,k_n)
$$

\end{itemize}
}
\end{defn}

\subsection{\sc The polymonoidal (op)lax-functor associated with an 
$n$-fold monoidal $\k$-linear category}
Let $\mathscr{C}$ be a $\k$-linear strict $n$-fold monoidal dg category. We assign to it a strict monoidal lax-functor
$$F_{\mathscr{C}}\colon (\Delta_f^\op)^{\times n}\to \mathscr{Cat}^\dg(\k)$$ and a strict monoidal oplax-functor
$$G_{\mathscr{C}}\colon (\Delta_f^\op)^{\times n}\to \mathscr{Cat}^\dg(\k)$$
Both functors are defined on objects as
\begin{equation}\label{street1}
F_{\mathscr{C}}([m_1],[m_2],\dots,[m_n])=G_{\mathscr{C}}([m_1],[m_2],\dots,[m_n])=\mathscr{C}^{\otimes (m_1\cdot m_2\cdot ...\cdot m_n)}
\end{equation}

Let $\alpha^i$ be a morphism in the $i$-th factor $\Delta_f^\op$ in the left-hand side of \eqref{street1}, corresponded to a map
$[\ell_i]\to [m_i]$ in $\Delta_f$. As in [BFSV, Section 2.1], we let $\alpha$ to act as 
\begin{equation}\label{street3}
F(\alpha^i)=(\alpha^i)^*\colon \mathscr{C}^{\otimes (m_1\dots m_{i}\dots m_n)}\to
\mathscr{C}^{\otimes (m_1\dots m_{i-1}\ell_i m_{i+1}\dots m_n)}
\end{equation}
as follows.

Denote by $\mathscr{A}=\mathscr{C}^{\otimes (m_{i+1}\cdot \dots\cdot m_n)}$. Regard $\mathscr{A}$ as a monoidal category, with the factor-wise monoidal product $\otimes_i$ (the $i$-th among the $n$ monoidal products which figure as a part of the structure of the $n$-fold monoidal category $\mathscr{C}$, see [BFSV, Def. 1.7]).

Then $\alpha^i$ defines (as for any monoidal category) a functor $F(\alpha^i)_+\colon \mathscr{A}^{\otimes m_i}\to \mathscr{A}^{\otimes \ell_i}$. Then the functor $F(\alpha^i)$ in \eqref{street3} is defined as 
\begin{equation}\label{street4}
F(\alpha^i)=(F(\alpha^i)_+)^{\otimes (m_1\cdot\dots\cdot m_{i-1})}
\end{equation}

Define 
\begin{equation}
G(\alpha^i)=F(\alpha^i)=(\alpha^i)^*
\end{equation}
by the same formula.

For a morphism $(\alpha_1,\alpha_2,\alpha_3,\dots,\alpha_n)$ in $(\Delta_f^\opp)^{\times n}$ 
define 
\begin{equation}\label{street5}
F((\alpha_1,\alpha_2,\alpha_3,\dots,\alpha_n))=F(\alpha_n^n)\circ\dots\circ F(\alpha_3^3)\circ F(\alpha_2^2)\circ F(\alpha^1_1)
\end{equation}
and
\begin{equation}\label{street5bis}
G((\alpha_1,\alpha_2,\alpha_3,\dots,\alpha_n))=G(\alpha_1^1)\circ\dots\circ G(\alpha_{n-2}^{n-2})\circ G(\alpha_{n-1}^{n-1})\circ G(\alpha_n^n)
\end{equation}
(We need to order the actions of morphisms acting in different factors $\Delta_f^\opp$ as they commute in $(\Delta_f^\opp)^n$ but their actions by \eqref{street4} do not commute). 

[BFSV, Section 2] constructs a morphism on functors 
\begin{equation}\label{street10}
F(\alpha^i)\circ F(\beta^j)\to F(\beta^j)\circ F(\alpha^i)
\end{equation}
{\it defined for $i<j$}, and for any $\alpha,\beta\in \Delta_f^\opp$, by making use of the Eckmann-Hilton maps $\eta^{ij}$.

In fact, [BFSV] constructs only the lax-functor $F_{\mathscr{C}}$; in virtue of \eqref{street10}, the oplax-functor $G_\mathscr{C}$ can be constructed simply by inverting the order of the composition, see \eqref{street5} and \eqref{street5bis}.

We can easily see it for $n=2$.
Let $\alpha=(\alpha_1,\alpha_2)$, $\beta=(\beta_1,\beta_2)$; then $\beta\alpha=(\beta_1\alpha_1,\beta_2\alpha_2)$.
Using more unified notation $(\alpha^i)^*=F(\alpha^i)=G(\alpha^i)$, one has:
$$
F(\alpha)=(\alpha^2_2)^*(\alpha_1^1)^*,\ \ F(\beta)=(\beta_2^2)^*(\beta_1^1)^*,\ \ F(\beta\alpha)=
(\beta_2^2\alpha_2^2)^*(\beta_1^1\alpha_1^1)^*=(\beta_2^2)^*(\alpha_2^2)^*(\beta_1^1)^*(\alpha_1^1)^*
$$
$$
G(\alpha)=(\alpha_1^1)^*(\alpha_2^2)^*,\ \ G(\beta)=(\beta_1^1)^*(\beta_2^2)^*,\ \ G(\alpha\beta)=
(\beta_1^1\alpha^1_1)^*(\beta_2^2\alpha_2^2)^*=(\beta_1^1)^*(\alpha_1^1)^*(\beta_2^2)^*(\alpha_2^2)^*
$$
(we made use that both $F$ and $G$ are strictly functorial for the composition of two morphisms acting on a fixed factor $i$).

Then \eqref{street10} gives morphisms of functors
\begin{equation}\label{street11}
(\alpha_1^1)^*(\beta_2^2)^*\to (\beta_2^2)^*(\alpha_1^1)^*,\ \ (\beta_1^1)^*(\alpha_2^2)^*\to (\alpha_2^2)^*(\beta_1^1)
\end{equation}
which give rise to morphisms
\begin{equation}\label{street12}
F(\beta)F(\alpha)\to F(\beta\alpha),\ \ G(\beta\alpha)\to G(\beta)G(\alpha)
\end{equation}

\begin{theorem}[cf. {[BFSV, Theorem 2.1]}]\label{monquasi}
Let $\mathscr{C}$ be a {\rm strict} $n$-fold monoidal dg category over $\k$. Then $\mathscr{C}$ gives rise to a strict poly-monoidal lax-functor 
$$
F_{\mathscr{C}}\colon (\Delta_f^\op)^{\times n}\to \mathscr{Cat}^\dg(\k)
$$

and to a strict poly-monoidal oplax-functor
$$
G_{\mathscr{C}}\colon (\Delta_f^\op)^{\times n}\to \mathscr{Cat}^\dg(\k)
$$
whose underlying (op)lax-functors are defined as above. 
The 2-arrows of the underlying lax-functor $F_\mathscr{C}$ and of the underlying oplax-functor $G_\mathscr{C}$ are equal to compositions of the Eckmann-Hilton maps $\eta_{ij}$.
\end{theorem}
See Definition \ref{polymon} for strict poly-monoidal (op)lax-functors. 
\begin{proof}
The statements that the constructed assignments are (op)lax-functors follow from the coherence theorem for $n$-fold monoidal categories, see [BFSV, Theorem 3.6], similarly to [BFSV, Theorem 2.1].

The statement that the (op)lax-functors $F_{\mathscr{C}}$ and $G_{\mathscr{C}}$ are strict poly-monoidal is equivalent to the commutativity of the diagram \eqref{polyeq}, which follows immediately from \eqref{street5} and \eqref{street5bis}.

The monoidal 2-category $\Cat^\dg(\k)$ (as well as the monoidal 2-category $\Cat$) is not strictly associative, and its associativity constrains should be implemented in the definition of a strict monoidal functor in the statement of theorem. It does not affect the statement, see Remark \ref{remgpsbis}.
\end{proof}

\section{\sc Deligne conjecture for essentially small $n$-fold monoidal abelian categories, $n\ge 1$}\label{deln}

\subsection{\sc Formulation of the result}
\begin{defn}\label{defn41n}{\rm
Let $\mathscr{A}$ be an abelian category, with an $n$-fold monoidal structure $(\otimes_1,\dots,\otimes_n;\ e)$ on it, see [BFSV, Section 1].
Denote by $\mathscr{A}^\dg$ the dg category of bounded from above complexes in $\mathscr{A}$, and let $\mathscr{J}\subset \mathscr{A}^\dg$ be the full dg subcategory of acyclic objects. Consider the full additive subcategory $\mathscr{A}_0\subset\mathscr{A}$ where $X\in\mathscr{A}_0$ iff
$X\otimes_i J$ and $J\otimes_i X$ are acyclic, for any acyclic $J\in\mathscr{J}$, and for any $1\le i\le n$. Denote by $\mathscr{A}_0^\dg\subset\mathscr{A}^\dg$ the full dg subcategory of bounded from above complexes in $\mathscr{A}_0$. We say that the exact and the monoidal structures in $\mathscr{A}$ are {\it weakly compatible} if the natural embedding
$$
\mathscr{A}_0^\dg\hookrightarrow \mathscr{A}^\dg
$$
is a quasi-equivalence of dg categories.
}
\end{defn}
The next definition did not appear in our discussion of the case $n=1$.
\begin{defn}\label{deff}{\rm
Let $\mathscr{A}$ be an abelian $\k$-linear category, with a weakly compatible $n$-fold monoidal structure on it.
We say that the $n$-fold monoidal structure is {\it non-degenerate} if for any four objects $X,Y,Z,W$ in $\mathscr{A}_0^\dg$, and for any $1\le i<j\le n$, the $(i,j)$-th Eckmann-Hilton map $\eta_{ij}\colon (X\otimes_j Y)\otimes_i (Z\otimes_j W)\to (X\otimes_i Z)\otimes_j (Y\otimes_i W)$ is a closed morphism of degree 0 and becomes an
isomorphism in the homotopy category $H^0(\mathscr{A}_0^\dg)$.
}
\end{defn}

Now we are ready to formulate our version of Deligne conjecture for $n$-fold monoidal abelian categories.

\begin{theorem}\label{delignesimplen}
Let $\k$ be a field of characteristic 0, and let $\mathscr{A}$ be essentially small $\k$-linear abelian category, endowed with $\k$-linear $n$-fold monoidal structure. Suppose the abelian and the $n$-fold monoidal structures are weakly compatible, see Definition \ref{defn41n}, and suppose that the $n$-fold monoidal structure is non-degenerate, see Definition \ref{deff}. Then $\RHom^\udot_{\mathscr{A}}(e,e)$ has a natural structure of a Leinster $(n,1)$-algebra over $\k$, whose underlying strict 1-algebra structure is given by the Yoneda product. (See Definition \ref{leinsterrelaxed} for weak Leinster $(n,1)$-monoids). More precisely, there is a weak Leinster $(n,1)$-monoid $X$ in $\Vect(\k)$, given by a functor $$X\colon (\Delta_f^\opp)^{\times n}\to\Alg(\k)$$
such that the underlying dg algebra over $\k$ $X_{1,\dots,1}$ is quasi-isomorphic to $\RHom^\udot_{\mathscr{A}}(e,e)$ with its Yoneda product.
\end{theorem}

Theorem \ref{delignesimplen} is proven in Section \ref{delignenproof} below.

\subsection{\sc Proof of Theorem \ref{delignesimplen}}\label{delignenproof}
Let $\mathscr{C}$ be an $n$-fold monoidal dg category over $\k$.

Recall the poly-monoidal oplax-functor $G_\mathscr{C}\colon (\Delta_f^\opp)^{\times n}\to 
\Cat^\dg(\k)$, given in Theorem \ref{monquasi}, with 
\begin{equation}\label{l3}
G_\mathscr{C}([m_1]\times\dots\times [m_n])=\mathscr{C}^{\otimes(m_1\dots m_n)}
\end{equation}
(It fails to be a honest functor from $(\Delta^\opp)^{\times n}\to\mathscr{C}at(\k)$ (where $\mathscr{C}at(\k)$ stands for the category of small $\k$-linear categories), as well as in the set-enriched case, see the discussion above Theorem \ref{monquasi}).

Let now $\mathscr{A}$ be a $\k$-linear abelian category with an $n$-fold monoidal structure on it.
Suppose that the abelian and the $n$-fold monoidal structures are weakly compatible, see Definition \ref{defn41n}, and suppose that the corresponding $n$-fold monoidal dg category $\mathscr{A}_0^\dg$ is non-degenerate, see Definition \ref{deff}.

Define the corresponding to the $n$-fold monoidal category $\mathscr{A}_0^\dg$ strict poly-monoidal oplax-functor
\begin{equation}
G=G_{\mathscr{A}_0^\dg}\colon (\Delta_f^\opp)^{\times n}\to\mathscr{C}at^\dg_-(\k)
\end{equation}
\begin{equation}
G([m_1]\times\dots\times [m_n])=(\mathscr{A}_0^\dg)^{\otimes(m_1\cdot\dots\cdot m_n)}
\end{equation}

We replace the involved essentially small dg categories by their small dg subcategories, as in Section \ref{section33}, and use the same notations for the corresponding small dg categories.

Lift the strict poly-monoidal oplax-functor $G\colon (\Delta_f^\opp)^{\times n}\to\mathscr{C}at^\dg_-(\k)$, to the strict poly-monoidal oplax-functor
$\hat{G}\colon (\Delta_f^\opp)^{\times n}\to \mathscr{PC}at^\dg_-(\k)$, as in our proof of $n=1$ case, see \eqref{who1}, \eqref{who2}.

More precisely, denote $\mathscr{J}_0=\mathscr{A}_0\cap\mathscr{J}$, see Definition \ref{defn41n}. Then
\begin{equation}
\hat{G}([m_1]\times[m_2]\times\dots\times[m_n])=\left((\mathscr{A}_0^\dg)^{\otimes (m_1\cdot\dots\cdot 
m_n)};\ \{\mathscr{C}_{i_1,\dots,i_n}^{[m_1,\dots,m_n]}\}_{1\le i_1\le m_1,\dots,1\le i_n\le m_n}\right)
\end{equation}
with
\begin{equation}
\mathscr{C}_{i_1,\dots,i_n}^{[m_1,\dots,m_n]}=\mathscr{C}_{i_1}^{[m_1]}\otimes \mathscr{C}_{i_2}^{[m_2]}\otimes\dots\otimes \mathscr{C}_{i_n}^{[m_n]}
\end{equation}
where $\mathscr{C}_{i_j}^{[m_j]}$ are defined as in \eqref{who2}.

Recall that $\mathscr{PC}at_-^\dg$ is a monoidal category of tuples of dg categories, see Section \ref{section33}.

Consider the composition
\begin{equation}\label{composa}
(\Delta_f^\opp)^{\times n}\xrightarrow{\hat{G}} \mathscr{PC}at^\dg_-(\k)\xrightarrow{\mathscr{D}r} \mathscr{C}at^\dg_-(\k)\xrightarrow{\mathscr{H}om}
\mathscr{M}on(\k)
\end{equation}
where $\mathscr{D}r$ is the refined Drinfeld dg quotient, constructed in Sections \ref{section330}, \ref{section33}, and $\mathscr{H}om(\mathscr{C})=\Hom_\mathscr{C}(*,*)$. 

Here, as well as in the proof of Theorem \ref{delignesimple}, we make use of the fact that the categories $\mathscr{D}r\circ \hat{G}([m_1]\times\dots\times [m_n])$ are ``based'', with the based objects (denoted by $*$) equal to the image of $[0]\times\dots\times[0]$
by the unique degeneracy morphism $[m_1]\times\dots \times [m_n]\to [0]\times\dots\times [0]$
in $\Delta_f^{\times n}$.

\begin{klemma}
The composition
$$
\mathscr{H}om\circ\mathscr{D}r\circ \hat{G}([1]\times\dots\times [1])\simeq \RHom^\udot_\mathscr{A}(e,e)
$$
as a dg algebra over $\k$.
\end{klemma}
The dg category $\hat{G}([1]\times[1]\times\dots \times [1])=\mathscr{A}_0$.
The rest is analogous to the proof of Key-Lemma \ref{kld}.

\qed

To complete the proof of Theorem \ref{delignesimplen}, we need to argue that the composition 
\eqref{composa} gives a weak Leinster $(n,1)$-monoid in $\Vect(\k)$. 
In fact, the partial composition of the first two functors $\mathscr{D}r\circ\hat{G}$ is a strict colax-polymonoidal oplax-functor, as it is a composition of the strict-polymonoidal oplax-functor $\hat{G}$ with the colax-monoidal strict functor $\mathscr{D}r$. 

Now we have a ``based'' strict colax-monoidal oplax-functor. Applying the $\Hom(*,*)$-functor to it, we get straightforwardly a weak Leinster $(n,1)$-monoid. (In fact, we composed the definition of the latter concept having this example in mind).

Theorem \ref{delignesimplen} is proven.

\section{\sc An application: the Gerstenhaber-Schack complex of a Hopf algebra}\label{stetr}
For any associative bialgebra $B$ over $\k$, there is a concept of {\it a tetramodule} over $B$.
Tetramodules form an abelian $\k$-linear category, denoted by $\Tetra(B)$. We proved in [Sh1], [Sh3] that $\Tetra(B)$ has a natural structure of a 2-fold monoidal category, in sense of [BFSV].
The deformation complex of a bialgebra $B$, the Gerstenhaber-Schack complex $C^\udot_\GS(B)$, can be intrinsically defined as $\RHom$ in the category $\Tetra(B)$:
$$
C^\udot_\GS(B)=\RHom^\udot_{\Tetra(B)}(B,B)
$$
Here we prove the following
\begin{theorem}\label{3alg}
Suppose $B$ is a Hopf algebra over $\k$ (a bialgebra over $\k$ with antipode). Then the 2-fold monoidal abelian category $\Tetra(B)$ satisfies the assumptions of Definitions \ref{defn41n} and \ref{deff}. More precisely, we can take for the additive subcategory $\mathscr{A}_0\subset\mathscr{A}=\Tetra(B)$ which figures in Definitions \ref{defn41n} and \ref{deff}, the entire abelian category $\Tetra(B)$. In particular, Theorem \ref{delignesimplen} is applicable to this category, and $\RHom^\udot_{\Tetra(B)}(B,B)$ has a structure of a Leinster 3-algebra over $\k$.
\end{theorem}
We proof the validation of the assumption of Definition \ref{defn41n} for $\mathscr{A}_0=\Tetra(B)$ in Proposition \ref{goedp}, and the validation of
the assumption of Definition \ref{deff} in Theorem \ref{goed} below. More precisely, we prove in Proposition \ref{goedp} that for a Hopf bialgebra $B$, the both monoidal products $\otimes_1$ and $\otimes_2$ are exact bi-functors, and that the Eckmann-Hilton map $\eta_{MNPQ}$ is an isomorphism for any $M,N,P,Q\in\Tetra(B)$. That is, what we prove below gives even stronger conditions than those of Definitions \ref{defn41n} and \ref{deff}.

\qed
\subsection{\sc Definitions}
Recall that {\it an associative bialgebra} over $\k$ is a $\k$-vector space $B$, endowed with a product $B\otimes B\to B$, a coproduct $\Delta\colon B\to B\otimes B$, a unit $i\colon k\to B$, a counit $\varepsilon\colon B\to k$ such that:
\begin{itemize}
\item[(i)] $(m,i)$ defines the structure of an associative algebra with unit on $B$,
\item[(ii)] $(\Delta,\varepsilon)$ defines a structure of a coassociative coalgebra with counit on $B$,
\item[(iii)] the compatibility: $\Delta(a*b)=\Delta(a)*\Delta(b)$ (here $a*b=m(a,b)$),
\item[(iv)] the counit is an algebra map, the unit is a coalgebra map.
\end{itemize}

Recall that an associative bialgebra is called a {\it Hopf algebra}, if there exists a $\k$-linear {\it antipode} map $S\colon B\to B$, satisfying the following properties:
\begin{itemize}
\item[(i)] $S$ is a linear isomorphism,
\item[(ii)] $m(1\otimes S)\Delta(x)=m(S\otimes 1)\Delta(x)=i(\varepsilon(x))$
\end{itemize}
One can deduce from this definition that
\begin{itemize}
\item[(iii)] $S(a*b)=S(b)*S(a)$, $\Delta(S(x))=(S\otimes S)(\Delta^\op(x))$,
\item[(iv)] $\varepsilon(S(x))=\varepsilon(x)$, $S(i(1))=i(1)$.
\end{itemize}

Recall that a {\it tetramodule} over a bialgebra $B$ is a $\k$-vector space $M$ such that $(B\oplus \epsilon M)[\epsilon]/(\epsilon^2)$ is once again an associative bialgebra, over $k[\epsilon]/(\epsilon^2)$, such that the canonical maps $B[\epsilon]/(\epsilon^2)\to (B\oplus \epsilon M)[\epsilon]/(\epsilon^2)$ and
$(B\oplus \epsilon M)[\epsilon]/(\epsilon^2)\to B[\epsilon]/(\epsilon^2)$ are bialgebra maps. It results to four structures:
\begin{itemize}
\item[T(i)] a left $B$-module structure $m_\ell\colon B\otimes M\to M$,
\item[T(ii)] a right $B$-module structure $m_r\colon M\otimes B\to B$,
\item[T(iii)] a left comodule structure $\Delta_\ell\colon M\to B\otimes M$,
\item[T(iv)] a right comodule structure $\Delta_r\colon M\to M\otimes B$
\end{itemize}
subject to the following compatibilities:
\begin{itemize}
\item[TC(i)] equipped with $(m_\ell,m_r)$, $M$ is a bimodule,
\item[TC(ii)] equipped with $(\Delta_\ell,\Delta_r)$, $M$ is a bi-comodule,
\item[TC(iii)] four ``bialgebra compatibilities'':
\end{itemize}
\begin{equation}\label{eq4.10}
\Delta_\ell(a*m)=(\Delta^1(a)*\Delta_\ell^1(m))\otimes
(\Delta^2(a)*\Delta^2_\ell(m))\subset B\otimes_k M
\end{equation}

\begin{equation}\label{eq4.11}
\Delta_\ell(m*a)=(\Delta_\ell^1(m)*\Delta^1(a))\otimes
(\Delta_\ell^2(m)*\Delta^2(a))\subset B\otimes_k M
\end{equation}

\begin{equation}\label{eq4.12}
\Delta_r(a*m)=(\Delta^1(a)*\Delta_r^1(m))\otimes
(\Delta^2(a)*\Delta^2_r(m))\subset  M\otimes_k B
\end{equation}

\begin{equation}\label{eq4.13}
\Delta_r(m*a)=(\Delta_r^1(m)*\Delta^1(a))\otimes
(\Delta^2_r(m)*\Delta^2(a))\subset M\otimes_k B
\end{equation}

As well, we see that the underlying $\k$-vector space $B$ is a $B$-tetramodule. It is the unit of the two-fold monoidal structure on $\Tetra(B)$, constructed in [Sh1], [Sh3].

The category of tetramodules is very important because of its relation to the Gerstenhaber-Schack complex [GS], the ``deformation complex'' of an associative bialgebra. The following result is due to R.Taillefer [Ta1,2] (see also an overview of Taillefer's results in [Sh1, Section 3]).
\begin{prop}
Let $B$ be an associative algebra over a field $\k$. Then the Gerstenhaber-Schack complex of $B$ is quasi-isomorphic to $\RHom^\udot_{\Tetra(B)}(B,B)$.
\end{prop}
\qed

\subsection{\sc Tetramodules over a Hopf algebra}
Recall that a Hopf module over a bialgebra $B$ over $\k$ is a $\k$-vector spaces $M$, endowed with a left $B$-module structure $m_\ell\colon B\otimes M\to M$, with a left $B$-comodule structure $\Delta_\ell\colon M\to B\otimes M$, such that \eqref{eq4.10} holds. In particular, any tetramodule over $B$ defines an underlying Hopf module over $B$.

The proof of Theorem \ref{3alg} is based on the following classical result, see [Sw, Theorem 4.1.1]:
\begin{klemma}\label{keywho}
\
\begin{itemize}
\item[(i)]
Let $B$ be a Hopf algebra, and let $M$ be a left Hopf module over $B$. Then $M$ is free as left $B$-module and is cofree as left $B$-comodule.
An analogous claim is true also when $M$ is a right Hopf module over $B$. More specifically, let $M$ be a left Hopf module over $B$, denote by $\Delta_\ell\colon M\to B\otimes M$ its comodule map. Denote
$$
M_\Delta=\{m\in M|\Delta_\ell(m)=1\otimes m\}
$$
Then the map of left $B$-modules $$\alpha\colon B\otimes_k M_\Delta\to M$$
$b\otimes m^\prime\mapsto b\cdot m^\prime$, is an isomorphism of Hopf modules. Analogously for right Hopf modules.
\item[(ii)] The inverse to the map $\alpha$ is a map
$$
\beta\colon M\to B\otimes M_\Delta
$$
is given by
\begin{equation}\label{keywho1}
\beta(m)=(\id_B\otimes P)\circ \Delta_\ell(m)
\end{equation}
where $P$ is the composition
\begin{equation}\label{keywho2}
P\colon M\xrightarrow{\Delta_\ell}B\otimes M\xrightarrow{S\otimes\id_M}B\otimes M\xrightarrow{m_\ell}M
\end{equation}
and, in fact,
\begin{equation}\label{keywho3}
P(M)=M_\Delta
\end{equation}
\end{itemize}
\end{klemma}
\qed
\begin{remark}
{\rm
In Key-Lemma above, it is very essential that $B$ is a Hopf algebra. The claim fails when $B$ is a general associative bialgebra.
}
\end{remark}

\begin{coroll}\label{coroll2014.1}
Let $B$ be a Hopf algebra, and $M$ a $B$-tetramodule. Then $M$ is in particular a left Hopf $B$-module and a right Hopf $B$-module.
Therefore, we have isomorphisms:
\begin{equation}\label{eq2014.1}
B\otimes_k M_{\Delta_\ell}\xrightarrow{\alpha_\ell}M\xleftarrow{\alpha_r}M_{\Delta_r}\otimes_k B
\end{equation}
where
\begin{equation}
M_{\Delta_\ell}=\{m\in M|\Delta_\ell(m)=1\otimes m\},\ \ M_{\Delta_r}=\{m\in M|\Delta_r(m)=m\otimes 1\}
\end{equation}
Furthermore, one can invert $\alpha_\ell$ and $\alpha_r$ explicitly, with $\beta_{\ell}=\alpha_{\ell}^{-1}$ and $\beta_r=\alpha_r^{-1}$ given by
\begin{equation}
\beta_\ell(m)=(\id_B\otimes P_\ell)\circ \Delta_\ell(m),\ \ \beta_r(m)=(P_r\otimes \id_B)\circ \Delta_r(m)
\end{equation}
where
\begin{equation}\label{eqplr}
P_\ell\colon M\xrightarrow{\Delta_\ell}B\otimes M\xrightarrow{S\otimes\id_M}B\otimes M\xrightarrow{m_\ell}M,\ \
P_r\colon M\xrightarrow{\Delta_r}M\otimes B\xrightarrow{\id_M\otimes S}M\otimes B\xrightarrow{m_r}M
\end{equation}
The left module structure $m_\ell$ and the left comodule structure $\Delta_\ell$ can be recovered from the leftmost term of \eqref{eq2014.1} as the product and the coproduct of $B$ (acting as identity on $M_{\Delta_\ell}$), the right module structure $m_r$ and the right comodule structure $\Delta_r$ can be recovered from the rightmost term of \eqref{eq2014.1} as the product and the coproduct of $B$ (acting as identity on $M_{\Delta_r}$).
\end{coroll}
\begin{proof}
The statements that $(m_\ell,\Delta_\ell)$ defines a left Hopf $B$-module on $M$, and that $(m_r,\Delta_r)$ defines a right Hopf $B$-module on $M$, are straightforward (see, however, Remark \ref{remark2014.1}). The remaining statements follow directly from Key-Lemma \ref{keywho}.
\end{proof}

\begin{remark}\label{remark2014.1}
{\rm
The category of tetramodules over $B$ fails to be the category of left Hopf modules over $B\otimes B^\opp$. Indeed, for $M$ a left Hopf module over $B\otimes B^\opp$, $M$ is endowed with structures of left and right $B$-modules (denote them by $m_\ell$ and $m_r$), and by left and right $B$-comodules (denote them by $\Delta_\ell$ and $\Delta_r$). For these 4 structures, $(m_\ell,\Delta_\ell)$ and $(m_r,\Delta_r)$ are compatible as the corresponding structures for a tetramodule, that is, as in \eqref{eq4.10} and \eqref{eq4.13}, correspondingly. However, two other tetramodule compatibilities \eqref{eq4.11} and \eqref{eq4.12} fail, as $m_\ell$ {\it commutes} with $\Delta_r$, and $m_r$ {\it commutes} with $\Delta_\ell$. Yet another way to see it is that the tautological tetramodule $B$ is not of the form $B\otimes B^\opp\otimes V$ for a vector space $V$.
}
\end{remark}

\begin{remark}
{\rm
It was mentioned to the author by V.Hinich that the results of P.Schauenburg [Scha] may imply that for the case of Hopf algebras $B$, the category of $B$-tetramodules is equivalent to the category of left Yetter-Drinfeld $B$-modules, where for a Yetter-Drinfeld module $L$ the underlying vector space of the corresponding tetramodule is $L\otimes B$. As the category of Yetter-Drinfeld modules is braided monoidal, it is expected that the category $\Tetra(B)$ is equivalent to the category of Yetter-Drinfeld modules over $B$, where the braiding on $\Tetra(B)$ follows from Theorem \ref{goed} and the Joyal-Street Theorem \ref{jst}.
}
\end{remark}

We develop the formalism of representing of a tetramodule $M$ over a Hopf algebra $B$ as \eqref{eq2014.1} a bit further, proving the following Lemma.
\begin{lemma}
Let $B$ be a Hopf algebra, $M\in \Tetra(B)$. The operators $P_\ell$ and $P_r$ introduced in \eqref{eqplr} obey the following identities:
\begin{equation}
P_\ell(b\cdot m)=\varepsilon(b)\cdot P_\ell(m)
\end{equation}
\begin{equation}
P_r(m\cdot b)=\varepsilon(b)\cdot P_r(m)
\end{equation}
\begin{equation}
P_\ell(b\cdot m)=S(\Delta^{(1)}b)\cdot P_\ell(m)\cdot \Delta^{(2)}b
\end{equation}
\begin{equation}
P_r(m\cdot b)=(\Delta^{(1)}b)\cdot P_r(m)\cdot S(\Delta^{(2)}b)
\end{equation}
for any $m\in M, b\in B$,
where we use the Sweedler notation $\Delta b=\Delta^{(1)}b\otimes \Delta^{(2)}b$.

\end{lemma}
\begin{proof}
\end{proof}

\begin{lemma}
Let $B$ be a Hopf algebra, $M\in\Tetra(B)$, $m_{\Delta_\ell}\in M_{\Delta_\ell}=\{m\in M|\Delta_{\ell}m=1\otimes m\}$.
Then, for any $b\in B$, the element
\begin{equation}
P_\ell(m_{\Delta_\ell}\cdot b)=S(\Delta^{(1)}b)\cdot m_{\Delta_\ell}\cdot \Delta^{(2)}b
\end{equation}
and
\begin{equation}
\Delta_r^{(1)}(m_{\Delta_\ell})
\end{equation}
belong to $M_{\Delta_\ell}$ as well, where the right coaction is $\Delta_rm_{\Delta_\ell}=\Delta_r^{(1)}m_{\Delta_\ell}\otimes \Delta_r^{(2)}m_{\Delta_\ell}\in M\otimes B$.

Analogously, if $m_{\Delta_r}\in M_{\Delta_r}=\{m\in M|\Delta_r(m)=m\otimes 1\}$, and $b\in B$, the element
\begin{equation}
P_r(b\cdot m_{\Delta_r})=\Delta^{(1)}(b)\cdot m_{\Delta_r}\cdot S(\Delta^{(2)}b)
\end{equation}
and
\begin{equation}
\Delta_\ell^{(2)}m_{\Delta_r}
\end{equation}
belong to $M_{\Delta_r}$ as well, where the left action is $\Delta_\ell(m_{\Delta_r})=\Delta_\ell^{(1)}(m_{\Delta_r})\otimes \Delta_\ell^{(2)}(m_{\Delta_r})\in B\otimes M$.
\end{lemma}

\begin{lemma}\label{lemmagreat}
Let $B$ be a Hopf algebra, $M\in\Tetra(B)$.
Then there are two decompositions
\begin{equation}\label{eqleftdec}
M=B\otimes_kM_{\Delta_\ell}
\end{equation}
and
\begin{equation}\label{eqrightdec}
M=M_{\Delta_r}\otimes B
\end{equation}
In the decomposition \eqref{eqleftdec}, the left action $m_\ell$ and the left coaction $\Delta_\ell$ act on the first factor $B$ as
\begin{equation}\label{eqgreat1}
b^\prime\cdot(b\otimes m_{\Delta_\ell})=(b^\prime\cdot b)\otimes m_{\Delta_\ell},\ \ \Delta_\ell(b\otimes m_{\Delta_\ell})=\Delta^{(1)}(b)\otimes(\Delta^{(2)}(b)\otimes m_{\Delta_\ell})
\end{equation}
the right action $m_r$ is
\begin{equation}\label{eqgreat2}
(b\otimes m_{\Delta_\ell})\cdot b^\prime=(b\cdot \Delta^{(1)}b^\prime)\otimes \left(S(\Delta^{(2)}b^\prime)\cdot m_{\Delta_\ell}\cdot \Delta^{(3)}b^\prime\right)
\end{equation}
and the right coaction $\Delta_r$ is
\begin{equation}\label{eqgreat3}
\Delta_r(b\otimes m_{\Delta_\ell})=\left(\Delta^{(1)}b\otimes\Delta^{(1)}_rm_{\Delta_\ell}\right)\otimes (\Delta^{(2)}b\cdot \Delta_r^{(2)}m_{\Delta_\ell})
\end{equation}
In the decomposition \eqref{eqrightdec}, the right action $m_r$ and the right coaction $\Delta_r$ act on the second factor $B$ as
\begin{equation}\label{eqgreat4}
(m_{\Delta_r}\otimes b)\cdot b^\prime=m_{\Delta_r}\otimes (b\cdot b^\prime),\ \ \Delta_r(m_{\Delta_r}\otimes b)=(m_{\Delta_r}\otimes\Delta^{(1)}b)\otimes\Delta^{(2)}b
\end{equation}
the left action $m_\ell$ is
\begin{equation}\label{eqgreat5}
b^\prime\cdot (m_{\Delta_r}\otimes b)=\left(\Delta^{(1)}b^\prime\cdot m_{\Delta_r}\cdot S(\Delta^{(2)}b^\prime)\right)\otimes (\Delta^{(3)}b^\prime\cdot b)
\end{equation}
and the left coaction $\Delta_\ell$ is
\begin{equation}\label{eqgreat6}
\Delta_\ell(m_{\Delta_r}\otimes b)=(\Delta^{(1)}_\ell m_{\Delta_r}\cdot \Delta^{(1)}b)\otimes \left(\Delta_\ell^{(2)}m_{\Delta_r}\otimes\Delta^{(2)}b\right)
\end{equation}
\end{lemma}

\smallskip

To continue with a proof of Theorem \ref{3alg}, we recall the construction of the 2-fold monoidal structure on $\Tetra(B)$.

\subsection{\sc The 2-fold monoidal structure on $\Tetra(B)$}\label{last}
Recall some constructions of [Sh1].
Let $B$ be an associative bialgebra, and let $M,N$ be two tetramodules over it.

One firstly define two their
``external'' tensor products $M\boxtimes_1N$ and
$M\boxtimes_2N$ (which are $B$-tetramodules once again). In both
cases the underlying vector space is $M\otimes_kN$. The tetramodule structures are defined as follows (where $a\in B$, $m\in M$, $n\in N$):

\smallskip

{\it The case of $M\boxtimes_1N$}:
\begin{equation}\label{thefirstprod}
\begin{aligned}
\ \ \ \ \ \ \ \ \ \ \ \ \ \ & m_\ell(a\otimes m\boxtimes n)=(am)\boxtimes n\\
&m_r(m\boxtimes n\otimes a)=m\boxtimes (na)\\
&\Delta_\ell(m\boxtimes
n)=(\Delta_\ell^1(m)*\Delta_\ell^1(n))\otimes
(\Delta^2_\ell(m)\boxtimes\Delta^2_\ell(n))\\
&\Delta_r(m\boxtimes n)=(\Delta^1_r(m)\boxtimes
\Delta^1_r(n))\otimes (\Delta^2_r(m)*\Delta^2_r(n))
\end{aligned}
\end{equation}

\smallskip

{\it The case of $M\boxtimes_2 N$}:
\begin{equation}\label{thesecondprod}
\begin{aligned}
\ &m_\ell(a\otimes m\boxtimes
n)=(\Delta^1(a)m)\boxtimes (\Delta^2(a)n)\\
&m_r(m\boxtimes n\otimes
a)=(m\Delta^1(a))\boxtimes(n\Delta^2(a))\\
&\Delta_\ell(m\boxtimes
n)=\Delta^1_\ell(m)\otimes(\Delta_\ell^2(m)\boxtimes n)\\
&\Delta_r(m\boxtimes
n)=(m\boxtimes\Delta_r^1(n))\otimes\Delta_r^2(n)
\end{aligned}
\end{equation}

\smallskip

Next, one defines
\begin{equation}\label{mon1}
M\otimes_1N=M\boxtimes_1N/\{\sum_i(m_ia)\boxtimes_1n_i-\sum_im_i\boxtimes_1(an_i),\ a\in B\}
\end{equation}
and
\begin{equation}\label{mon2}
M\otimes_2N=\left\{\sum_im_i\boxtimes_2 n_i\subset
M\boxtimes_2 N|\sum_i \Delta_r(m_i)\otimes_k
n_i=\sum_im_i\otimes_k\Delta_\ell(n_i)\right\}
\end{equation}

In [Sh1], we constructed for any four $M,N,P,Q\in\Tetra(B)$ the {\it Eckman-Hilton map}
\begin{equation}
\eta_{MNPQ}\colon (M\otimes_2 N)\otimes_1(P\otimes_2Q)\to (M\otimes_1P)\otimes_2(N\otimes_1Q)
\end{equation}
which is proven to satisfy all necessary commutative diagrams (see [BFSV], Section 1) making $\Tetra(B)$ a 2-fold monoidal category.

In particular, the tautological tetramodule $B\in\Tetra(B)$ is the two-sided unit for both $\otimes_1$ and $\otimes_2$:
\begin{equation}
B\otimes_1M=M\otimes_1B=B\otimes_2M=M\otimes_2B=M
\end{equation}
for any $M\in\Tetra(B)$.

For further reference, we summarize in Lemma below some properties of the map $\eta_{MNPQ}$, proven in [Sh1], Section 2.2.3.

\begin{lemma}\label{lemma2014.la}
The map $\eta_{MNPQ}$ is induced by the map
\begin{equation}\label{eq2014.la}
\hat{\eta}_{MNPQ}\colon (M\boxtimes_2N)\boxtimes_1(P\boxtimes_2Q)\to (M\boxtimes_1P)\boxtimes_2(N\boxtimes_1Q)
\end{equation}
defined on the underlying vector spaces as
\begin{equation}\label{eq2014.la2}
m\otimes_k n\otimes_k p\otimes_k q\mapsto m\otimes_k p\otimes_k n\otimes_k q
\end{equation}
for $m,n,p,q$ elements of $M,N,P,Q$, correspondingly. By ``induced'' is meant the following: starting with $\hat{\eta}_{MNPQ}$, we firstly consider the projection of the rhs of \eqref{eq2014.la} to $(M\otimes_1P)\boxtimes_2(N\otimes_1Q)$ and show that the composition of $\hat{\eta}_{MNPQ}$ with this projection descents to a well-defined map
\begin{equation}\label{eq2014.la3}
\hat{\hat{\eta}}_{MNPQ}\colon (M\boxtimes_2N)\otimes_1(P\boxtimes_2Q)\to (M\otimes_1P)\boxtimes_2(N\otimes_1Q)
\end{equation}
Nextly, we restrict the lhs of \eqref{eq2014.la3} to its subspace $(M\otimes_2N)\otimes_1(P\otimes_2Q)$, and show that the image of this subspace by
$\hat{\hat{\eta}}_{MNPQ}$ belongs to $(M\otimes_1P)\otimes_2(N\otimes_1Q)$ (which is a subspace of the rhs of \eqref{eq2014.la3}). The resulting map
$$
(M\otimes_2N)\otimes_1(P\otimes_2Q)\to (M\otimes_1P)\otimes_2(N\otimes_1Q)
$$
is the map $\eta_{MNPQ}$.
\end{lemma}

\qed

\begin{lemma}\label{prowho}
Let $B$ be a Hopf algebra over $\k$, $M,N$ two $B$-tetramodules. Then both monoidal products
$M\otimes_1 N$ and $M\otimes_2N$ have isomorphic underlying vector space, isomorphic to
\begin{equation}\label{eq2014.u1}
M_{\Delta_r}\otimes_k B\otimes_k N_{\Delta_\ell}
\end{equation}
(in notations of Corollary \ref{coroll2014.1}).
The projection
\begin{equation}\label{eq2014.u2}
M\boxtimes_1N=(M_{\Delta_r}\otimes_kB)\otimes_k(B\otimes_k N_{\Delta_\ell})\to M_{\Delta_r}\otimes_k B\otimes_k N_{\Delta_\ell}=M\otimes_1N
\end{equation}
is given by the product map $B\otimes_kB\to B$ in the middle, and the identity on the leftmost and the rightmost terms. The inclusion
\begin{equation}\label{eq2014.u3}
M\otimes_2N=M_{\Delta_r}\otimes_k B\otimes_k N_{\Delta_\ell}\hookrightarrow (M_{\Delta_r}\otimes_kB)\otimes_k(B\otimes_k N_{\Delta_\ell})=M\boxtimes_2N
\end{equation}
is given by the coproduct $\Delta\colon B\to B\otimes_k B$ in the middle, and the identity on the leftmost and the rightmost terms.
\end{lemma}

\begin{proof}
We use the presentations $M=M_{\Delta_r}\otimes_k B$ and $N=B\otimes_kN_{\Delta_\ell}$, given by Corollary \ref{coroll2014.1}.
In these presentations, we can recover $\Delta_r$ and $m_r$ for $M$, and $\Delta_\ell$ and $m_\ell$ for $N$, as the coproduct and the product on $B$.

Now, by \eqref{mon1}, the equation \eqref{eq2014.u2} follows as the product $m\colon B\otimes_k B\to B$ is surjective (as $B$ contains unit).
To deduce \eqref{eq2014.u3} from \eqref{mon2}, we need to know that the kernel of the map
$$
d\colon B\otimes_k B\to B\otimes_kB\otimes_kB
$$
defined by
$$
d(b_1\otimes b_2)=\Delta(b_1)\otimes b_2-b_1\otimes \Delta(b_2)
$$
is the image of the coproduct $\Delta\colon B\to B\otimes_k B$. The latter follows from the acyclicity of the cobar-complex of any coalgebra with counit.
\end{proof}

We can prove now the first part of Theorem \ref{3alg}.
\begin{prop}\label{goedp}
Let $B$ be a Hopf algebra over $\k$. Then both monoidal products $\otimes_1$ and $\otimes_2$ on $\Tetra(B)$ are exact bi-functors.
\end{prop}
\begin{proof}
It follows from Lemma \ref{lemmagreat} and Lemma \ref{prowho}. By Lemma \ref{lemmagreat}, any tetramodule over a Hopf algebra $B$ is a free left $B$-module,
free right $B$-module, cofree left $B$-comodule, and cofree right $B$-comodule.
Then Lemma \ref{prowho} shows that it implies the exactness of $\otimes_1$ and of $\otimes_1$ on the level of underlying vector spaces, and therefore their exactness as bi-functors on the category of tetramodules.
\end{proof}

We pass now to study of the Eckmann-Hilton map $\eta_{MNPQ}$ for the category $\Tetra(B)$, where $B$ is a Hopf algebra.
The second part of Theorem \ref{3alg} is proven in Theorem \ref{goed} at the end of this Section.

\begin{prop}\label{propo}
Let $B$ be a Hopf algebra over $\k$, $M,N$ two $B$-tetramodules. Then the tetramodules
$M\otimes_1 N$ and $M\otimes_2N$ are isomorphic. Using the vector space isomorphisms of both $M\otimes_1N$ and $M\otimes_2N$ to
$M_{\Delta_r}\otimes_kB\otimes_kN_{\Delta_\ell}$ from Lemma \ref{prowho}, the identity map
\begin{equation}
\varphi=\id\colon M_{\Delta_r}\otimes_kB\otimes_kN_{\Delta_\ell}\to M_{\Delta_r}\otimes_kB\otimes_kN_{\Delta_\ell}
\end{equation}
defines an isomorphism of tetramodules
$$
\varphi\colon M\otimes_1N\to M\otimes_2N
$$
\end{prop}
\begin{proof}
We give two different proofs of the Proposition, both of which are instructive.

\smallskip

{\it The first proof:}

Consider general 2-fold monoidal category $\mathscr{C}$ with unit $e$ (see [BFSV], Section 1), with the Eckmann-Hilton map
$$
\eta_{MNPQ}\colon (M\otimes_2 N)\otimes_1(P\otimes_2Q)\to (M\otimes_1P)\otimes_2(N\otimes_1Q)
$$
(where $M,N,P,Q\in\mathscr{C}$). It is a morphism in $\mathscr{C}$. Take $N=P=e$, then we get the morphism
\begin{equation}\label{eq2014.u7}
\eta_{MeeQ}\colon M\otimes_1Q\to M\otimes_2Q
\end{equation}
It is also a morphism in $\mathscr{C}$.

For the case $\mathscr{C}=\Tetra(B)$, we prove that this morphism $\eta_{MeeQ}$ is an isomorphism (where $e=B$ is the tautological tetraodule).

We know (see Lemma \ref{lemma2014.la}) that the map $\eta_{MNPQ}$ is induced by the map
$\hat{\eta}_{MNPQ}$ which is just the transposition of the two middle factors, see \eqref{eq2014.la2}.
In the same time, we want to use the presentation for the underlying vector space \eqref{eq2014.u1} we just found.
Our goal is to prove that \eqref{eq2014.u7} is an {\it isomorphism of vector spaces} (because it is a map of tetramodules by the above general argument).

The diagram below is not commutative, but becomes commutative after passing $\boxtimes_i\to\otimes_i$ ($i=1,2$):
\begin{equation}
\xymatrix{
&((M_{\Delta_r}\otimes_kB)\boxtimes_2B)\boxtimes_1(B\boxtimes_2(B\otimes_kN_{\Delta_\ell}))\ar[dd]^{{\hat{\eta}}_{MBBN}}\\
M_{\Delta_r}\otimes_kB\otimes_kN_{\Delta_\ell}\ar[ur]^{f_1}\ar[dr]_{f_2}\\
&((M_{\Delta_r}\otimes_kB)\boxtimes_1B)\boxtimes_2(B\boxtimes_1(B\otimes_kN_{\Delta_\ell}))
}
\end{equation}
where
\begin{equation}
\begin{aligned}
\ &f_1(m\otimes b \otimes n)=((m\otimes 1)\boxtimes_21)\boxtimes_1(\Delta^1(b)\boxtimes_2(\Delta^2(b)\otimes n))\\
&f_2(m\otimes b^\prime\otimes n)=((m\otimes 1)\boxtimes_1\Delta^1(b^\prime))\boxtimes_2(\Delta^2(b^\prime)\boxtimes_1(1\otimes n))
\end{aligned}
\end{equation}
where $m\in M_{\Delta_r}$, $n\in N_{\Delta_\ell}$, $b\in B$.
The maps $f_1,f_2$ are compatible with the vector space isomorphisms of Lemma \ref{prowho}.

\qed

\smallskip

{\it The second proof:}

We write down explicitly the tetramodule structures on $M\otimes_1N$ and $M\otimes_2N$ identifying the underlying vector spaces with
$M_{\Delta_r}\otimes_k B\otimes_k N_{\Delta_ell}$ as in Lemma \ref{prowho}. We use for that explicit formulas found in Lemma \ref{lemmagreat}.

We will show that the left actions are equal and that the left coaction are equal for $M\otimes_1N$ and $M\otimes_2N$; the case of right actions and right coactions goes similarly.

We use the following isomorphisms of $M\otimes_1N$ and of $M\otimes_2N$ with $M_{\Delta_r}\otimes_kB\otimes_kN_{\Delta_r}$:
\begin{equation}\label{twostrings}
\xymatrix{
&&(m_{\Delta_r}\otimes b)\boxtimes_1(1\otimes n_{\Delta_\ell})\subset M\boxtimes_1N\\
m_{\Delta_r}\otimes b\otimes n_{\Delta_\ell}\ar[urr]^{i_1}\ar[drr]_{i_2}\\
&&(m_{\Delta_r}\otimes\Delta^{(1)}b)\boxtimes_2(\Delta^{(2)}b\otimes n_{\Delta_\ell})\subset M\boxtimes_2 N
}
\end{equation}

\smallskip

{\it The case of $M\otimes_1N$:}

For the left action, one has:
\begin{equation}
\begin{aligned}
\ &a\cdot((m_{\Delta_r}\otimes b)\boxtimes_1 (1\otimes n_{\Delta_r}))\overset{\text{by \eqref{thefirstprod}}}=(a\cdot (m_{\Delta_r}\otimes b))\boxtimes_1(1\otimes n_{\Delta_r})\overset{\text{by \eqref{eqgreat1}}}{=}\\
&\left((\Delta^{(1)}a\cdot m_{\Delta_r}\cdot S(\Delta^{(2)}a)\otimes \Delta^{(3)}a\cdot b\right)\boxtimes_1(1\otimes n_{\Delta_{\ell}})=\\
&
i_1\left((\Delta^{(1)}a\cdot m_{\Delta_r}\cdot S(\Delta^{(2)}a))\otimes (\Delta^{(3)}a\cdot b)\otimes n_{\Delta_\ell})\right)
\end{aligned}
\end{equation}
where $i_1$ is the upper arrow in diagram \eqref{twostrings}.

For the left coaction, one has:
\begin{equation}
\begin{aligned}
\ &\Delta_\ell((m_{\Delta_r}\otimes b)\boxtimes_1 (1\otimes n_{\Delta_r}))\overset{\text{by \eqref{thefirstprod}}}=\\
&\left(\Delta_\ell^{(1)}(m_{\Delta_r}\otimes b)\cdot \Delta_\ell^{(1)}(1\otimes n_{\Delta_\ell})\right)\otimes \left(\Delta_\ell^{(2)}(m_{\Delta_r}\otimes b)\boxtimes_1 \Delta_\ell^{(2)}(1\otimes n_{\Delta_\ell})\right)\overset{\text{by \eqref{eqgreat3}}}=\\
&(\Delta_\ell^{(1)}m_{\Delta_r}\cdot \Delta^{(1)}b\cdot 1)\otimes \left((\Delta_\ell^{(2)}m_{\Delta_r}\otimes\Delta^{(2)}b)\boxtimes_1(1\otimes n_{\Delta_\ell})\right)=\\
&(\Delta_\ell^{(1)}m_{\Delta_r}\cdot \Delta^{(1)}b)\otimes i_1\left(\Delta_\ell^{(2)}m_{\Delta_r}\otimes\Delta^{(2)}b\otimes n_{\Delta_\ell}\right)
\end{aligned}
\end{equation}

\smallskip

{\it The case of $M\otimes_2N$:}

For the left action, one has:
\begin{equation}
\begin{aligned}
\ &a\cdot\left((m_{\Delta_r}\otimes\Delta^{(1)}b)\boxtimes_2(\Delta^{(2)}b\otimes n_{\Delta_\ell})\right)\overset{\text{by \eqref{thesecondprod}}}=\\
&\Delta^{(1)}a\cdot(m_{\Delta_r}\otimes\Delta^{(1)}b)\boxtimes_2\Delta^{(2)}a\cdot(\Delta^{(2)}b\otimes n_{\Delta_\ell})\overset{\text{by \eqref{eqgreat2}}}=\\
&\left((\Delta^{(1)}a\cdot m_{\Delta_r}\cdot S(\Delta^{(2)}a))\otimes (\Delta^{(3)}a\cdot \Delta^{(1)}b)\right)\boxtimes_2
(\Delta^{(4)}a\cdot\Delta^{(2)}b\otimes n_{\Delta_\ell})=\\
&i_2\left((\Delta^{(1)}a\cdot m_{\Delta_r}\cdot S(\Delta^{(2)}a))\otimes (\Delta^{(3)}a\cdot b)\otimes n_{\Delta_\ell}\right)
\end{aligned}
\end{equation}
where $i_2$ is the lower arrow in diagram \eqref{twostrings}.

For the left coaction, one has:
\begin{equation}
\begin{aligned}
\ &\Delta_\ell\left((m_{\Delta_r}\otimes\Delta^{(1)}b)\boxtimes_2(\Delta^{(2)}b\otimes n_{\Delta_\ell})\right)\overset{\text{by \eqref{thesecondprod}}}=\\
&\Delta_\ell^{(1)}(m_{\Delta_r}\otimes\Delta^{(1)}b)\otimes\left(\Delta_{\ell}^{(2)}(m_{\Delta_r}\otimes\Delta^{(1)}b)\boxtimes_2(\Delta^{(2)}b\otimes n_{\Delta_\ell})\right)\overset{\text{by \eqref{eqgreat4}}}=\\
&(\Delta_\ell^{(1)}m_{\Delta_r}\cdot \Delta^{(1)}b)\otimes\left((\Delta^{(2)}_\ell m_{\Delta_r}\otimes\Delta^{(2)}b)\boxtimes_2(\Delta^{(3)}b\otimes n_{\Delta_\ell})\right)=\\
&(\Delta_\ell^{(1)}m_{\Delta_r}\cdot \Delta^{(1)}b)\otimes i_2\left(\Delta^{(2)}_\ell m_{\Delta_r}\otimes \Delta^{(2)}b\otimes n_{\Delta_\ell}\right)
\end{aligned}
\end{equation}
We see that the left action of $a\in B$ on an element of $M_{\Delta_r}\otimes_k B\otimes_k N_{\Delta_\ell}$ is $i_1(X)$ when
$M_{\Delta_r}\otimes_k B\otimes_k N_{\Delta_\ell}$ is considered as $M\otimes_1N$, and is $i_2(X)$, {\it for the same $X$}, when
$M_{\Delta_r}\otimes_k B\otimes_k N_{\Delta_\ell}$ is considered as $M\otimes_2 N$; the similar result holds for the left coaction(s).

The case of the right action and the right coaction is similar.
\end{proof}

\begin{theorem}\label{goed}
Let $B$ be a Hopf algebra over $\k$, $M,N,P,Q$ be any four $B$-tetramodules.
Then the Eckmann-Hilton map $\eta_{MNPQ}$ is an isomorphism.
\end{theorem}
\begin{proof}
We start with a Lemma:
\begin{lemma}\label{lemmawd}
Let $B$ be a Hopf algebra, $M,N$ be $B$-tetramodules. Consider the ``right form'' of them, see Lemma \ref{lemmagreat}:
\begin{equation}
M=M_{\Delta_r}\otimes_k B,\ \ N=N_{\Delta_r}\otimes_k B
\end{equation}
Then the ``right form'' of the tetramodule $M\otimes_1N$ is
\begin{equation}
M\otimes_1N=(M_{\Delta_r}\otimes_k N_{\Delta_r})\otimes_k B
\end{equation}
with the standard right action and the standard right coaction (acting only on the rightmost factor $B$), the left action given by
\begin{equation}\label{wq1}
a\cdot(m_{\Delta_r}\otimes n_{\Delta_r}\otimes b)=(\Delta^{(1)}a\cdot m_{\Delta_r}\cdot S(\Delta^{(2)}a))\otimes (\Delta^{(3)}a\cdot n_{\Delta_r}\cdot S(\Delta^{(4)}a))\otimes (\Delta^{(5)}a\cdot b)
\end{equation}
and the left coaction given by
\begin{equation}\label{wq2}
\Delta_\ell(m_{\Delta_r}\otimes n_{\Delta_r}\otimes b)=\left(\Delta_\ell^{(1)}m_{\Delta_r}\cdot \Delta^{(1)}_\ell n_{\Delta_r}\cdot \Delta^{(1)}b\right)\otimes_k\left(\Delta_\ell^{(2)}m_{\Delta_r}\otimes\Delta_\ell^{(2)}n_{\Delta_r}\otimes\Delta^{(2)}b\right)
\end{equation}
The map
\begin{equation}\label{pick1}
\vartheta_{MN,r}\colon(M_{\Delta_r}\otimes_k B)\otimes_1 (N_{\Delta_r}\otimes_k B)\to (M_{\Delta_r}\otimes_k N_{\Delta_r})\otimes_k B
\end{equation}
\begin{equation}\label{pick2}
\vartheta_{MN,r}((m_{\Delta_r}\otimes b_1)\otimes_1 (n_{\Delta_r}\otimes b_2))=\left(m_{\Delta_r}\otimes (\Delta^{(1)}b_1\cdot n_{\Delta_r}\cdot S(\Delta^{(2)}b_1))\right)
\otimes (\Delta^{(3)}b_1\cdot b_2)
\end{equation}
is a map of tetramodules. The map $\theta_{MN,r}$ is an isomorphism for any $M,N$.

There are analogous statements for the ``left form'' presentations $M=B\otimes_k M_{\Delta_\ell}$, $N=B\otimes_k N_{\Delta_\ell}$, and their product $M\otimes_1 N$.
\end{lemma}
\begin{proof}
By the definition of $M\otimes_1N$ as the quotient of $M\boxtimes_1N$, see \eqref{mon1}, we have the following identity in $M\otimes_1N$:
\begin{equation}\label{wd1}
\begin{aligned}
\ &\left((m_{\Delta_r}\otimes 1)\cdot b_1\right)\otimes_1(n_{\Delta_r}\otimes b_2)=(m_{\Delta_r}\otimes 1)\otimes_1 \left(b_1\cdot(n_{\Delta_r}\otimes b_2)\right)\overset{\text{by \eqref{eqgreat5}}}=\\
&(m_{\Delta_r}\otimes 1)\otimes_1\left((\Delta^{(1)}b_1\cdot n_{\Delta_r}\cdot S(\Delta^{(2)}b_1))\otimes (\Delta^{(3)}b_1\cdot b_2)\right)
\end{aligned}
\end{equation}
It is clear therefore that the map \eqref{pick2} is an isomorphism of vector spaces.

It only remains to deduce the tetramodule structure from the one on the left-hand side of \eqref{wd1}.
Here we use, first of all, the formulas \eqref{thefirstprod} for the tetramodule structure $M\boxtimes_1N$.

Then we find for the left action:
\begin{equation}\label{wd2}
\begin{aligned}
\ &a\cdot\left((m_{\Delta_r}\otimes b_1)\otimes_1 (n_{\Delta_r}\otimes b_2)\right)\overset{\text{by \eqref{thefirstprod}}}=
\left(a\cdot (m_{\Delta_r}\otimes b_1)\right)\otimes_1(n_{\Delta_r}\otimes b_2)\overset{\text{by \eqref{eqgreat5}}}=\\
&\left((\Delta^{(1)}a\cdot m_{\Delta_r}\cdot S(\Delta^{(2)}a))\otimes (\Delta^{(3)}a\cdot b_1)\right)\otimes_1 (n_{\Delta_r}\otimes b_2)\overset{\text{by \eqref{wd1}}}=\\
&\left((\Delta^{(1)}a\cdot m_{\Delta_r}\cdot S(\Delta^{(2)}a))\otimes 1\right)\otimes_1\left((\Delta^{(1)}(\Delta^{(3)}a\cdot b_1)\cdot n_{\Delta_r}\cdot S(\Delta^{(2)}(\Delta^{(3)}a\cdot b_1)))\otimes (\Delta^{(3)}(\Delta^{(3)}a\cdot b_1)\cdot b_2)\right)=\\
&\left((\Delta^{(1)}a\cdot m_{\Delta_r}\cdot S(\Delta^{(2)}a))\otimes 1\right)\otimes_1\left(\left(\Delta^{(4)}a\cdot \Delta^{(1)}b_1\cdot n_{\Delta_r}\cdot S(\Delta^{(2)}b_1)\cdot S(\Delta^{(4)}a)\right)\otimes(\Delta^{(5)}a\cdot \Delta^{(3)}b_1\cdot b_2)\right)
\end{aligned}
\end{equation}
It follows from \eqref{wd2} that, within the identification \eqref{wd1}, one has:
\begin{equation}\label{wd3}
a\cdot(m_{\Delta_r}\otimes n_{\Delta_r}\otimes b)=(\Delta^{(1)}a\cdot m_{\Delta_r}\cdot S(\Delta^{(2)}a))\otimes (\Delta^{(3)}a\cdot n_{\Delta_r}\cdot S(\Delta^{(4)}a))\otimes (\Delta^{(5)}a\cdot b)
\end{equation}
It is the formula \eqref{wq1} for the left action.

The formula \eqref{wq2} for the left coaction is deduced analogously.

Finally, the only elements satisfying $\Delta_r(X)=X\otimes 1$ are the linear combinations of the elements
$m_{\Delta_r}\otimes n_{\Delta_r}\otimes 1$. That is, \eqref{pick1} and \eqref{pick2} give indeed the ``left form'' presentation for the teramodule $M\otimes_1 N$.
\end{proof}
We pass now to the proof of Theorem.

Consider the map
\begin{equation}\label{wj1}
\eta_{MNPQ}\colon (M\otimes_2N)\otimes_1(P\otimes_2Q)\to (M\otimes_1P)\otimes_2(N\otimes_1 Q)
\end{equation}
We use the presentation $M_{\Delta_r}\otimes_k B\otimes_k N_{\Delta_\ell}$ for $M\otimes_2 N$, and
the presentation $P_{\Delta_r}\otimes_k B\otimes_k Q_{\Delta_\ell}$ for $P\otimes_2 Q$, with the corresponding isomorphisms \eqref{eq2014.u3}.

That is, a general element in $M\otimes_2N$ is (a linear combination of the elements) $(m_{\Delta_r}\otimes \Delta^{(1)}b_1)\otimes_2(\Delta^{(2)}b_1\otimes m_{\Delta_\ell})$, and a general element in $P\otimes_2 Q$ is (a linear combination of the elements)
$(p_{\Delta_r}\otimes \Delta^{(1)}b_2)\otimes_2(\Delta^{(2)}b_2\otimes q_{\Delta_\ell})$.

Due to the $\otimes_1$-product, we can assume that $b_2=1$.
Thus, the lhs of \eqref{wj1} has form
$M_{\Delta_r}\otimes_k B\otimes_k N_{\Delta_\ell}\otimes_k P_{\Delta_r}\otimes_k Q_{\Delta_\ell}$, and the corresponding isomorphism is
\begin{equation}
m_{\Delta_r}\otimes b\otimes n_{\Delta_\ell}\otimes p_{\Delta_r}\otimes q_{\Delta_\ell}\mapsto \left((m_{\Delta_r}\otimes\Delta^{(1)}b)\otimes_2 (\Delta^{(2)}b\otimes n_{\Delta_\ell})\right)\otimes_1 \left((p_{\Delta_r}\otimes 1)\otimes_2 (1\otimes q_{\Delta_\ell})\right)
\end{equation}
The map $\eta_{MNPQ}$, due to its description in Lemma \ref{lemma2014.la}, acts as
\begin{equation}
\begin{aligned}
\ &\left((m_{\Delta_r}\otimes\Delta^{(1)}b)\otimes_2 (\Delta^{(2)}b\otimes n_{\Delta_\ell})\right)\otimes_1 \left((p_{\Delta_r}\otimes 1)\otimes_2 (1\otimes q_{\Delta_\ell})\right)\mapsto \\
&\left((m_{\Delta_r}\otimes\Delta^{(1)}b)\otimes_1 (p_{\Delta_r}\otimes 1)\right)\boxtimes_2 \left((\Delta^{(2)}b\otimes n_{\Delta_\ell})\otimes_1(1\otimes q_{\Delta_\ell})\right)
\end{aligned}
\end{equation}
We need to prove that this map is an isomorphism.

Now we use the isomorphisms
\begin{equation}
(M_{\Delta_r}\otimes_k B)\otimes_1(P_{\Delta_r}\otimes_k B)\to (M_{\Delta_r}\otimes_k P_{\Delta_r})\otimes_k B
\end{equation}
and
\begin{equation}
(B\otimes_k N_{\Delta_\ell})\otimes_1 (B\otimes_k Q_{\Delta_\ell})\to B\otimes_k(N_{\Delta_\ell}\otimes_k Q_{\Delta_\ell})
\end{equation}
given by Lemma \ref{lemmawd}.

This lemma establishes that these maps {\it are isomorphisms}, and they map
\begin{equation}\label{wdw}
\begin{aligned}
\ &\vartheta_{MP,r}((m_{\Delta_r}\otimes\Delta^{(1)}b)\otimes_1 (p_{\Delta_r}\otimes 1))=m_{\Delta_r}\otimes (\Delta^{(1)}b\cdot p_{\Delta_r}\cdot S(\Delta^{(2)}b))\otimes \Delta^{(3)}b\\
&\vartheta_{NQ,\ell}((\Delta^{(4)}b\otimes n_{\Delta_\ell})\otimes_1(1\otimes q_{\Delta_\ell}))=\Delta^{(4)}b\otimes n_{\Delta_\ell}\otimes q_{\Delta_\ell}
\end{aligned}
\end{equation}
Finally, the map $\eta_{MNPQ}$ acts as
\begin{equation}
m_{\Delta_r}\otimes b\otimes n_{\Delta_\ell}\otimes p_{\Delta_r}\otimes q_{\Delta_\ell}\mapsto
\left(m_{\Delta_r}\otimes  (\Delta^{(1)}b\cdot p_{\Delta_r}\cdot S(\Delta^{(2)}b))\right)\otimes \Delta^{(3)}b\otimes \left(n_{\Delta_\ell}\otimes q_{\Delta_\ell}\right)
\end{equation}
It is an isomorphism because the map $\vartheta_{MP,r}$ (see \eqref{wdw}) is an isomorphism by Lemma \ref{lemmawd}.

\end{proof}

\smallskip

The situation of the 2-fold monoidal category $\Tetra(B)$, where $B$ is a Hopf algebra, gives an illustration for the following result, due to Joyal and Street [JS]:

\begin{theorem}[Joyal-Street]\label{jst}
Suppose $\mathscr{C}$ be an $n$-fold monoidal category, for which all Eckmann-Hilton maps $\eta^{i,j}$, $1\le i<j\le n$, are isomorphisms.
Suppose $n=2$. Consider the map
\begin{equation}
\lambda_{MN}: M\otimes_1N\xrightarrow{\eta_{eMNe}}N\otimes_2M\xrightarrow{\eta_{NeeM}^{-1}}N\otimes_1M
\end{equation}
Then $(\mathscr{C},\otimes_1,\lambda)$ is a braided monoidal category.

Conversely, the 2-fold monoidal category $\mathscr{C}^\prime$ whose underlying category is that of $\mathscr{C}$, both monoidal products $\circ_1$ and $\circ_2$ are equal and equal to $\otimes_1$, and $\eta^\prime_{MNPQ}\colon (M\circ_2N)\circ_1(P\circ_2 Q)\to (M\circ_1P)\circ_2(N\circ_1Q)$ is defined as $\id_M\otimes_1\lambda_{NP}\otimes_1\id_Q$, is equivalent as a 2-fold monoidal category to $\mathscr{C}$.

When $n>2$, an $n$-fold monoidal category with all $\eta^{ij}_{MNPQ}$ isomorphisms, is a symmetric monoidal category.
\end{theorem}

Using our proof of Theorem \ref{goed}, we can compute the braiding $\lambda_{MN}$ explicitly in terms of the antipode $S$.

\bigskip
{\small
\noindent {\sc Universiteit Antwerpen, Campus Middelheim, Wiskunde en Informatica, Gebouw G\\
Middelheimlaan 1, 2020 Antwerpen, Belgi\"{e}}}

\bigskip

\noindent{{\it e-mail}: {\tt Boris.Shoikhet@uantwerpen.be}}

\end{document}